\documentclass[12pt]{article}
\usepackage[left=1in,top=1in,right=1in,bottom=1in]{geometry}

\usepackage[english]{babel}
\usepackage{amsmath,amsthm,amsfonts,amssymb,epsfig,xcolor,authblk, hyperref}
\usepackage{bbm,booktabs,float,mathtools,siunitx,tikz,ulem}

\newtheorem{thm}{Theorem}[section]
\newtheorem{lem}[thm]{Lemma}
\newtheorem{prop}[thm]{Proposition}
\newtheorem{cor}[thm]{Corollary}

\newtheorem*{rem}{Remark}

\newcommand{\E}{\mathbf{E}}

\newcommand{\PP}{\mathbf{P}}

\newcommand{\R}{\mathbb{R}}

\newcommand{\ii}{{\rm i}}
\newcommand{\oo}{{\rm o}}
\newcommand{\rd}{{\rm d}}
\newcommand{\m}{\text{meas}}

\DeclareMathOperator{\OO}{O}

\newcommand{\re}{{\rm Re}}

\newcommand{\1}{\mathbf 1}

\title{\Large Large Deviation Estimates of Selberg's Central Limit Theorem \\  and Applications}

\author[1,2]{\normalsize Louis-Pierre Arguin}
\author[2]{\normalsize Emma Bailey}

\affil[1]{\footnotesize  \it Department of Mathematics, Baruch College, CUNY, New York, NY}
\affil[2]{\footnotesize  \it Department of Mathematics, CUNY Graduate Center, New York, NY}

\date{February 20, 2022}
\begin{document}

\maketitle
\begin{abstract}
For $V\sim \alpha \log\log T$ with $0<\alpha<2$, we prove
\[
\frac{1}{T}\text{meas}\{t\in [T,2T]: \log|\zeta(1/2+ {\rm i} t)|>V\}\ll \frac{1}{\sqrt{\log\log T}} e^{-V^2/\log\log T}.
\]
This improves prior results of Soundararajan and of Harper on the large deviations of Selberg's Central Limit Theorem in that range, without the use of the Riemann hypothesis. 
The result implies the sharp upper bound for the fractional moments of the Riemann zeta function proved by Heap, Radziwi{\l}{\l} and Soundararajan. 
It also shows a new upper bound for the maximum of the zeta function on short intervals of length $(\log T)^\theta$, $0<\theta <3$, that is expected to be sharp for $\theta > 0$. 
Finally, it yields a sharp upper bound (to order one) for the moments on short intervals, below and above the freezing transition. 
The proof is an adaptation of the recursive scheme introduced by Bourgade, Radziwi{\l}{\l} and one of the authors to prove fine asymptotics for the maximum on intervals of length $1$.
\end{abstract}
%\tableofcontents

\section{Introduction}
\subsection{Main Result}
Selberg's Central Limit Theorem \cite{sel46, sel92} states that the logarithm of the Riemann zeta function $\zeta(s)$ at a typical point on the critical line $\re \ s=1/2$ behaves like a complex Gaussian random variable of mean $0$ and variance $\log\log T$. 
Specifically, if $\tau$ is uniformly distributed on $[T,2T]$, then for the real part of the logarithm we have
\[
\PP\Big(\log |\zeta(1/2+\ii \tau)|>\sqrt{\tfrac{1}{2}\log\log T}\cdot y\Big)\sim \int_y^{\infty} \frac{e^{-z^2/2}}{\sqrt{2\pi}}\rd z, \quad y\in \R, \text{ as $T\to\infty$}.
\]
See \cite{radsou17} for an elegant self-contained proof of this, and \cite{sou21} for a survey on the distribution of values of $L$-functions in general.
In this paper, we prove that the above Gaussian decay persists in the large deviation regime:
\begin{thm}
\label{thm: LD Selberg}
Let $V\sim \alpha \log\log T$ with $0<\alpha<2$. We have for $T$ large enough
\[
\PP(\log|\zeta(1/2+\ii \tau)|>V)\ll \frac{1}{\sqrt{\log\log T}}\exp\left(\frac{-V^2}{\log\log T}\right).
\]
The implicit constant in the inequality can be taken uniform in $\alpha$ in any compact subset of $(0,2)$.
\end{thm}
Throughout the paper, the notation $\ll$ means that the left is $\OO$ of the right side as $T\to\infty$, and that the implicit constant is possibly $\alpha$-dependent.

In the interval $0<V< 2\log\log T$, Theorem \ref{thm: LD Selberg} is an improvement of a more general theorem of Soundararajan \cite{sou09}, which states for this particular range that
\begin{equation}
\label{eqn: sound}
\PP(\log|\zeta(1/2+\ii \tau)|>V)\ll (\log T)^{\oo(1)}\cdot \exp\left(\frac{-V^2}{\log\log T}\right).
\end{equation}
Harper \cite{har13b} also proved sharp bounds for the moments of the zeta function, which by Markov's inequality imply 
\begin{equation}
\label{eqn: harper}
\PP(\log|\zeta(1/2+\ii \tau)|>V)\ll \exp\left(\frac{-V^2}{\log\log T}\right).
\end{equation}
Both results assume the Riemann hypothesis, whereas Theorem \ref{thm: LD Selberg} is unconditional.
Equations~\eqref{eqn: sound} and~\eqref{eqn: harper} do hold conditionally on a wider range of $V$, for example $V\sim k\log\log T$ for any $k>0$. 

Heap, Radziwi\l{\l} and Soundararajan proved sharp upper bounds for the moments between $0$ and $4$, cf.~Corollary \ref{cor: fractional}, which imply Equation \eqref{eqn: harper} unconditionally.
For $\sqrt{\log\log T}\log\log\log T\leq V\leq 2\log\log T-2\sqrt{\log\log T}\log\log\log T$, Heap and Soundararajan~\cite{heasou20} also proved unconditionally the asymptotic behavior 
\[\PP(\log|\zeta(1/2+\ii \tau)|>V)= \exp\left(\frac{-V^2}{\log\log T}+\OO\left(\frac{V\log\log\log T}{\sqrt{\log\log T}}\right)\right).\]
It was conjectured by Radziwi{\l}{\l} \cite{rad11} that the Gaussian behavior actually extends to the whole range $V\sim k\log\log T$, $k>0$, up to a multiplicative factor
\[
\PP\Big(\log |\zeta(1/2+\ii \tau)|>V\Big)\sim C_k \int_V^{\infty} \frac{e^{-y^2/\log\log T}}{\sqrt{\pi\log\log T}}\rd y,
\]
where $C_k$ is the conjectured leading coefficient of the $2k$-moment (cf.~\cite{keasna00a}).
If we write $V=\alpha \log\log T + \sigma y$ for $\sigma=\oo(\sqrt{\log\log T})$, then Theorem \ref{thm: LD Selberg} also gives an upper bound to order one for a local version of Selberg's Central Limit Theorem, as proposed in \cite{bormelnik19}. (See Proposition 4.8 there for a more precise result for a random model of zeta.) Finally, we also remark that for characteristic polynomials of random unitary matrices, large deviations in the equivalent regime to Theorem~\ref{thm: LD Selberg} were proved in~\cite{hugkeaoco01} and precise asymptotics (including the constant) were proved in~\cite{fermelnik16}.

Theorem \ref{thm: LD Selberg} is proved in Section \ref{sect: proof}. 
The method is an adaptation of a recursive scheme introduced in \cite{argbourad20} to prove a sharp upper bound to the Fyodorov-Hiary-Keating Conjecture, cf.~Equation \eqref{eqn: FHK}.
Consider the Dirichlet polynomials
\begin{equation}
  \label{eqn: S1}
  S_k=\sum_{\log{2}\leq\log{p} \leq e^k} \frac{\re\ p^{-\ii \tau}}{p^{1/2}}, \quad k\geq 1.
\end{equation}
These partial sums are a good proxy for $\log |\zeta(1/2+\ii \tau)|$ for $k$ close to $\log\log T$. 
Moreover, the moments of $S_k$ are very close to Gaussian, see for example Lemma \ref{lem: Gaussian moments real} or \cite[Lemma 3.4]{abbrs19}.
However, the error for these moments is too large to handle simultaneously $k$ close to $\log\log T$ as well as moments of order $\log\log T$. 

The idea is to restrict the estimate of the probability to good events where the partial sums~\eqref{eqn: S1} takes values in a narrow interval. 
The implementation of this recursive scheme is much simpler here than in \cite{argbourad20}, where restrictions at every $k$ were needed. 
Namely, for Theorem \ref{thm: LD Selberg}, the partial sums only need to be constrained on a sparse collection of $k$'s of the form
\begin{equation}
\label{eqn: t1}
t_\ell=\log\log T-\mathfrak s \log_{\ell+2} T, \quad \ell\geq 1,
\end{equation}
for some ($\alpha$-dependent) $\mathfrak s$, where $\log_\ell$ stands for the logarithm iterated $\ell$ times.
Moreover, since Theorem~\ref{thm: LD Selberg} only concerns large values of $\zeta$ at a single point, no discretization is needed here compared to~\cite{argbourad20} where the authors considered the maximum of $\zeta$ over a short $\OO(1)$-range.  This simplifies the statements and proofs of various foundational results (cf. Lemmas~\ref{lemma_4 easy},~\ref{lemma_3}, and \ref{lemma_4}) regarding second and twisted fourth moments of Dirichlet polynomials.  As a corollary to Theorem~\ref{thm: LD Selberg}, we prove an upper bound on the maximum of $\zeta$ over a growing window, cf. Corollary~\ref{cor: max}.

The restriction is on good events of the form 
\[
\{S_{t_\ell}\in [L_\ell, U_\ell] \}, \quad \ell\geq 1,
\]
where $L_\ell$ is slightly below the linear interpolation $\alpha t_\ell$ and $U_\ell$ is slightly above. 
These barriers must be chosen carefully and dependent on $\alpha$. Also, $U_\ell$ must be much higher than the upper barrier picked in \cite{argbourad20} as the fluctuations here can be greater.
It turns out that the dominant term of the probability in Theorem \ref{thm: LD Selberg} comes from the intersection of all the good events above. 
On these events, the increments $S_{t_{\ell+1}}-S_{t_{\ell}}$ are restricted to a range where large deviations can be estimated. 
%The appendix contains standard results on moments of Dirichlet polynomials and of a random model of the zeta function. 

Theorem \ref{thm: LD Selberg} must be restricted to $\alpha<2$ since we rely on a twisted fourth moment estimate (Lemma \ref{lemma_9}). 
More generally, large deviations in the range $\alpha\log\log T$ are controlled by the $2\alpha$-moment of zeta. 
This suggests that the method of proof should be adaptable to prove Theorem \ref{thm: LD Selberg} for any $\alpha>0$ assuming the Riemann hypothesis, where all such moments can be sharply bounded.
This would improve the bounds \eqref{eqn: sound} and \eqref{eqn: harper} in the full range $\alpha\log\log T$, $\alpha>0$, conditionally.
We also expect that a matching lower bound (up to constant) can be found using the techniques of \cite{argbourad22}. 
In \cite{rad11}, it was proved that Selberg's theorem holds up to $V$ of the order of $(\log\log T)^{3/5-\varepsilon}$.  Subsequently, Inoue~\cite{ino19} improved the range of $V$ up to $(\log\log T)^{2/3}$.  
The techniques involved in the proof of Theorem \ref{thm: LD Selberg} do not seem to be applicable to the range $V=\oo(\log\log T)$. 
Interestingly, this leaves a gap between $V\ll(\log\log T)^{2/3}$ and $V\sim \alpha\log\log T$ where the Gaussian decay remains open.

\subsection{Applications}
The first corollary of Theorem \ref{thm: LD Selberg} is an alternative proof of a sharp upper bound for fractional moments of the zeta function, proved unconditionally by Heap, Radziwi\l{\l} and Soundararajan.
\begin{cor}[Theorem 1 in \cite{hearadsou19}]
  \label{cor: fractional}
  Let $0< \beta < 4$. We have for $T$ large enough
  \begin{equation}
  \label{eqn: M}
  M_\beta=\frac{1}{T}\int_T^{2T} |\zeta(1/2+\ii t)|^\beta \rd t \ll(\log T)^{\beta^2/4},
  \end{equation}
  where the implicit constant depends on $\beta$.
\end{cor} 
The proof in \cite{hearadsou19} depends on twisted fourth moment estimates, as for Theorem \ref{thm: LD Selberg}. 
Hence, it might be considered at the same conceptual level as the proof of Corollary \ref{cor: fractional}.
Corollary \ref{cor: fractional} is proved in Section \ref{sect: fractional}.  Note that, via Markov's inequality, Equation~\eqref{eqn: M} shows in particular the Gaussian decay \eqref{eqn: harper} unconditionally, for $V\sim \frac{\beta}{2}\log\log T$ and $\beta\in[0,4]$.  

In short intervals, of size $(\log T)^\theta$ for $0\leq \theta<3$, Theorem \ref{thm: LD Selberg} implies an upper bound for the maximum up to order one precision:
\begin{cor}
  \label{cor: max}
  Let $0\leq \theta<3$ and $y>0$ such that $y=\oo\left(\frac{\log\log T}{\log\log\log T}\right)$. We have
  \begin{equation}
  \label{eqn: max}
  \max_{|h|\leq (\log T)^\theta}|\zeta(1/2+\ii t+\ii h)|\leq e^y\frac{(\log T)^{\sqrt{1+\theta}}}{(\log\log T)^{1/(4\sqrt{1+\theta})}},
  \end{equation}
  for all $t\in[T,2T]$ except on a set of Lebesgue measure $\ll e^{-2\sqrt{1+\theta}y} e^{-y^2/\log\log T}$.
\end{cor}
The restriction to $\theta<3$ is due to the limitations in the range of large deviation, up to $2\log\log T$, in Theorem \ref{thm: LD Selberg}.
The result also gives a precise decay for the right tail of the maximum, which is exponential for small $y$'s and Gaussian for large ones. The condition on the size of $y$ in the statement of the corollary can be relaxed at the expense of a different decay rate, as can be easily observed within the proof.
Upper and lower bounds for the maximum with error $(\log T)^\varepsilon$ were proved in \cite{argouirad19}.
Corollary \ref{cor: max} proves the fine asymptotics up to order one as given in Conjecture 1.3 of \cite{argouirad19}. 
The proof of Corollary \ref{cor: max} is given in Section \ref{sect: max}. It is a simple union bound after suitably discretizing the interval on $(\log T)^{1+\theta}$ points.
It is expected that the bound is sharp for $\theta>0$, see \cite{aabhr21} for numerical evidence of this.
This is because for $\theta>0$, the values of zeta at the $(\log T)^{1+\theta}$ points should each behave like IID Gaussians of variance $\frac{1}{2}\log\log T$, see for example \cite{argouirad19}
\footnote{Closely related is a class of models called `continuous random energy models', cf.~\cite{bokulo02, bovier06, bovier17, bovhar15} that exhibit similar extreme value statistics for a suitable choice of parameters.}.
This is in contrast with the case $\theta=0$. Corollary \ref{cor: max} holds for this case, but it is not sharp. It was conjectured by Fyodorov, Hiary \& Keating and Fyodorov \& Keating, that the maximum of $\log |\zeta|$ on intervals of size one should behave exactly like the maximum of log-correlated stochastic processes \cite{fyohiakea12, fyokea14}.
It was shown in \cite{argbourad20} that
\begin{equation}
  \label{eqn: FHK}
  \max_{|h|\leq 1}|\zeta(1/2+\ii t+\ii h)|\leq e^y\frac{\log T}{(\log\log T)^{3/4}},
\end{equation}
for all $t\in[T,2T]$ except on a set of Lebesgue measure $\ll y e^{-2y} e^{-y^2/\log\log T}$.
Upper and lower bounds with error $(\log T)^\varepsilon$ were proved in \cite{naj18, abbrs19}.
A hybrid regime interpolating between IID and log-correlated statistics was also proposed in \cite{argdubhar21}.
For more on recent developments in extreme values of log-correlated processes, see for example \cite{baikea22}.

 Theorem \ref{thm: LD Selberg} can also be applied to improve current bounds for the moments of $\zeta$ in short intervals.
\begin{cor}
  \label{cor: subcritical}
  Let $0\leq \theta<3$. For all $\beta\geq 0$, we have for $A>1$
  \begin{equation}
    \label{eqn: subcritical}
    \int_{|h|\leq (\log T)^\theta} |\zeta(1/2+\ii t+\ii h)|^\beta\rd h\leq A (\log T)^{\frac{\beta^2}{4}+\theta},
  \end{equation}
  for all $t\in [T,2T]$ except possibly on a subset of Lebesgue measure $\ll 1/A$. 
  
  For $\beta>\beta_c=2\sqrt{1+\theta}$, a sharper bound holds:
  \begin{equation}
    \label{eqn: supercritical}
    \int_{|h|\leq (\log T)^\theta} |\zeta(1/2+\ii t+\ii h)|^\beta\rd h\leq C_{A,\beta}\cdot  (\log\log T)^{-\tfrac{\beta}{2\beta_c}}\cdot (\log T)^{\tfrac{\beta_c}{2}\beta -1},
  \end{equation}
  for all $t\in [T,2T]$ except possibly on a subset of Lebesgue measure $\ll 1/A$, where $C_{A,\beta}$ is an explicit constant dependent on $A$ and $\beta$.
\end{cor}
Equation \eqref{eqn: subcritical} was proved in \cite{argouirad19}. It follows easily by Markov's inequality and the bound \eqref{eqn: M}.
Nevertheless, we provide another proof of this using the Lebesgue measure of high points.
This is helpful in understanding the proof of the sharper bound for the moments above $\beta_c$. 
Equation \eqref{eqn: supercritical} is an improvement on \cite{argouirad19}, where the result was given with a $(\log T)^\varepsilon$ error.
Interestingly, Equation \eqref{eqn: supercritical} is exactly the behavior expected for the moments of $(\log T)^{1+\theta}$ IID Gaussian random variables of variance $\frac{1}{2}\log\log T$ as computed by Bovier, Kurkova \& L\"owe \cite[Theorem 1.6]{bokulo02} for large $\beta$.  

Equations \eqref{eqn: subcritical} and \eqref{eqn: supercritical} exhibit a {\it freezing transition} (also referred to as {\it intermittency}) where the moments transition from quadratic to linear growth. In view of this, it is natural to ask if the bound \eqref{eqn: subcritical} at criticality $\beta=\beta_c$ is sharp. 
At $\theta=0$, where the system seems to behave like a log-correlated process, it can be improved as shown by Harper:
\begin{thm}[Theorem 1 and Corollary 1 in \cite{har19}]
\label{thm: Harper}
We have
\[
\int_{|h|\leq 1} |\zeta(1/2+\ii t+\ii h)|^2 \rd h\leq A \frac{\log T}{\sqrt{\log\log T}},
\]
for all $t\in [T,2T]$ except possibly on a subset of Lebesgue measure $\ll \frac{(\log A)\wedge \sqrt{\log\log T}}{A}$. 
\end{thm}
The presence of the correction $1/\sqrt{\log\log T}$ is related to the phenomenon of critical Gaussian multiplicative chaos, see \cite{pow18}.
In Section \ref{sect: critical}, we explain how this correction appears in view of the Lebesgue measure of high points.
For $\theta>0$, where the IID heuristic prevails, such a correction should be absent as predicted by Theorem 1.6 (i) of  \cite{bokulo02}.
Hence, Equation \eqref{eqn: subcritical} is expected to be sharp to order one at $\beta=\beta_c$.\\

\noindent{\bf Notation}. 
Throughout the proofs, we use the probabilistic convention for random variables and often drop the dependence on $\tau$, which will always be taken uniform on $[T, 2T]$, to lighten the notation.
Most dramatically, we will simply write 
\[
\text{$\zeta$ for the random variable $\zeta(1/2+\ii\tau)$.}
\]

Another convenient notation is
\[
t=\log\log T.
\]
It turns out that $\log\log$ is the correct scale for the primes in the considered problems.
This is because the Dirichlet sums considered,  see for example~\eqref{eqn: S1} and \eqref{eqn: S} below, behave like a random walk on that scale.\\

\noindent{\bf Acknowledgements}
We thank Paul Bourgade and Maksym Radziwi{\l\l} for insightful discussions on the subject. 
The research of LPA was supported in part by NSF CAREER.
DMS-1653602.  Part of this work was conducted whilst EB participated in a program during the Fall 2021 semester hosted by the Mathematical Sciences Research Institute in Berkeley, California, which was supported by the NSF Grant No.~DMS-1928930.

\section{Proof of Theorem \ref{thm: LD Selberg}}
\label{sect: proof}
The proof is an adaptation of the recursive scheme of \cite{argbourad20}. 
First, we introduce some notations.
Consider the partial Dirichlet sums
\begin{equation}
  \label{eqn: S}
  S_k=\sum_{2\leq p \leq \exp(e^k)} \frac{\re\ p^{-\ii \tau}}{p^{1/2}}+\frac{\re \ p^{-2\ii\tau}}{2p}, \quad k\geq 1,
\end{equation}
with $S_0=0$.
(As opposed to the simpler Equation \eqref{eqn: S1}, we include here the square of primes within the definition.
This simplifies the application of Lemma \ref{lemma_23} below.)
For $S_k$ to be a good approximation for $\log|\zeta|$, the parameter $k$ must be taken close to $t$. 
With this in mind, $t$ is approached in a finite number of steps by iterated logarithms as in \eqref{eqn: t1}:
\begin{equation}
\label{eqn: t}
t_\ell=t - \mathfrak{s} \log_\ell t, \quad \ell \geq 1,
\end{equation}
with the convention that $t_0=0$.
The parameter $\mathfrak s$ here depends on $\alpha$. 
A good choice (reflecting the symmetry in $\alpha$) is
\begin{equation}
\label{eqn: s}
\mathfrak{s}=\frac{2\cdot 10^6}{(2-\alpha)^2 \alpha^2}.
\end{equation}
We will say more on this choice below Equation \eqref{eqn: B restrict}.
The last $\ell$, denoted by $\mathcal L$, is defined as the largest $\ell$ such that
\begin{equation}
\exp(10^6(t-t_\ell)^{10^5}e^{t_{\ell+1}})\leq \exp\Big(\frac{1}{100}e^t\Big)=T^{1/100}.
\end{equation}
Note that the left-hand side is
\[
\exp\Big(10^6(\mathfrak s \log_\ell t)^{10^5}\cdot \frac{e^t}{(\log_\ell t)^\mathfrak s}\Big),
\]
therefore the choice of $\mathfrak s$ ensures that such a $\mathcal L$ exists if $T$ is large enough.
By definition, we also have $\log_{\mathcal L} t=\OO(1)$ and $\log_{\mathcal L} t>0$.
The corresponding complex partial sums are also needed and are denoted by
  \begin{equation}
  \label{eqn: S tilde}
  \widetilde{S}_k=\sum_{2\leq p \leq \exp(e^k)} \frac{\ p^{-\ii \tau}}{p^{1/2}}+\frac{ \ p^{-2\ii\tau}}{2p}, \quad k\geq 1,
  \end{equation}
  and $\widetilde S_0=0$.
We stress that only the values of the partial sums at $t_\ell$, $1\leq \ell\leq \mathcal L$, are necessary. 

To approximate $\exp(-S_{t_\ell})$, we use the mollifiers:
\begin{equation}\label{eq:mollifier}
\mathcal{M}_\ell=\sum_{\substack{p|m\implies \log\log p\in(t_{\ell-1},t_\ell]\\\Omega_\ell(m)\leq(t_\ell-t_{\ell-1})^{10^5}}}\frac{\mu(m)}{m^{\frac{1}{2}+i\tau}},
\end{equation}
where $\Omega_\ell(m)$ is the number of prime factors of $m$ in $(\exp(e^{t_\ell}),\exp(e^{t_{\ell+1}})]$ with multiplicity, and $\mu(m)$ is the M\"{o}bius function.
 The proof will show that product $\mathcal{M}_{1}\cdots\mathcal{M}_{\ell}$ is typically a good approximation for $\exp(-S_{t_\ell})$.\\

The idea of the proof is to partition the event
\[
H=\{\log |\zeta(1/2+\ii \tau)|>V\}
\]
into recursively defined events that greatly restrict the values of the Dirichlet sums~\eqref{eqn: S} and~\eqref{eqn: S tilde}. 
It is expected that, if $\log |\zeta(1/2+\ii \tau)|>V$ and $V\sim \alpha t$, then the partial sum $S_{t_\ell}$ should be close to $\kappa t_\ell$ where
\begin{equation}\label{eq:slope}
  \kappa=\frac{V}{t}\sim \alpha.
\end{equation}
More precisely, consider for $1\leq \ell\leq \mathcal L$, the decreasing events
\begin{equation}
\label{eqn: ABCD}
\begin{aligned}
  A_\ell&=A_{\ell-1}\cap\{|\widetilde{S}_{t_\ell}-\widetilde{S}_{t_{\ell-1}}|\leq \mathcal{A}(t_\ell-t_{\ell-1})\}\\
  B_\ell&=B_{\ell-1}\cap\{S_{t_\ell}\leq \kappa t_\ell + \mathcal{B} \log _\ell t \}\\
  C_\ell&=C_{\ell-1}\cap\{S_{t_\ell}\geq \kappa t_\ell -\mathcal{C}\log _\ell t\}\\
  D_\ell&=D_{\ell-1}\cap\{|\zeta e^{-S_{t_\ell}}|\leq c_\ell|\zeta\mathcal{M}_1\cdots \mathcal M_\ell|+e^{-\mathcal{D}(t-t_{\ell-1})}\},
\end{aligned}
\end{equation}
where $c_\ell=\prod_{j=1}^\ell(1+e^{-t_{j-1}})$, and $A_{0},B_{0},C_{0},D_{0}=[T,2T]$ (the full sample space).
The parameters $\mathcal{A}, \mathcal{B}, \mathcal{C}, \mathcal{D}$ will be chosen carefully as discussed below.
For now, we simply observe that on the {\it good} event
\[
G_\ell=A_\ell \cap B_\ell \cap C_\ell \cap D_\ell,
\]
the partial sums are restricted in a narrow corridor between an upper and lower barrier:
\begin{equation}\label{eq:barriers}
U_\ell=\kappa t_\ell + \mathcal{B} \log _\ell t\qquad L_\ell=\kappa t_\ell -\mathcal{C}\log _\ell t.
\end{equation}
The auxiliary event $D_\ell$ ensures that $\exp(-S_{t_\ell})$ is well approximated by the mollifier, and $A_\ell$ is an {\it a priori estimate} needed for the estimates involving $C_\ell$ and $D_\ell$. 
The probability of $H=\{\log |\zeta(1/2+\ii \tau)|>V\}$ can then be decomposed over the $G_\ell$'s. 
The dominant contribution comes from $H\cap G_{\mathcal{L}}$ where the sums are restricted up to order one away from $t$.
The precise estimates are: 

\begin{prop}
\label{prop: first}
Let $V\sim \alpha t$ with $0<\alpha<2$. With the notation above, we have for some $\delta>0$ (dependent on $\alpha$) and $t$ large enough
\[
\PP(H\cap G_1^c)\ll \frac{e^{-V^2/t}}{\sqrt{t}} \cdot t^{-\delta}.
\]
\end{prop}

\begin{prop}
\label{prop: induction}
Let $V\sim \alpha t$ with $0<\alpha<2$. With the notation above, we have for $1\leq \ell \leq \mathcal L-1$, some $\delta>0$ (dependent on $\alpha$ but not $\ell$), and $t$ large enough
\[
\PP(H\cap G_\ell\cap G_{\ell+1}^c)\ll \frac{e^{-V^2/t}}{\sqrt{t}} \cdot (\log_\ell t)^{-\delta}.
\]
\end{prop}

\begin{prop}
\label{prop: last}
Let $V\sim\alpha t$ with $0<\alpha<2$.  With the notation above, we have for $t$ large enough
\[
\PP(H\cap G_{\mathcal L})\ll\frac{1}{\sqrt{t}}e^{-V^2/t}.
\]
\end{prop}

The theorem is a simple consequence of the three propositions.
\begin{proof}[Proof of Theorem \ref{thm: LD Selberg}]
It suffices to notice that
\[
\PP(H)=\PP(H\cap G_1^c)+\sum_{\ell=1}^{\mathcal{L}-1} \PP(H\cap G_\ell\setminus G_{\ell+1})+\PP(H\cap G_{\mathcal{L}}).
\]
The result follows by applying Propositions \ref{prop: first}, \ref{prop: first}, \ref{prop: last}.
\end{proof}

As mentioned above, the parameters in \eqref{eqn: ABCD} need to be chosen in a delicate manner.
As we shall see from the proof (cf. Equations~\eqref{eqn: B const1} and~\eqref{eqn: B const2}), the choice of $\mathcal B$ must satisfy the following restrictions.
\begin{equation}
  \label{eqn: B restrict}
  \begin{aligned}
    1+\alpha^2 \mathfrak{s}-2\alpha \mathcal{B}&<0\\
    \mathcal{B}-\alpha \mathfrak{s}&<0.
  \end{aligned}
\end{equation}
The first equation forces $\mathcal B$ to be proportional to $1/\alpha$ to handle small $\alpha$'s.
In turn, the second equation leads to $\mathfrak s>1/\alpha^2$, motivating in part the choice of $\mathfrak s$ in \eqref{eqn: s}.
With this choice, the defining inequalities for $\mathcal B$ becomes 
\begin{align*}
  \frac{1}{2\alpha}+\frac{10^6}{\alpha(2-\alpha)^2}<&\mathcal{B}<\frac{10^6}{\alpha(2-\alpha)^2}+\frac{10^6}{\alpha(2-\alpha)^2}.
\end{align*}
This is a non-empty interval since $\alpha>0$.  Therefore, a valid choice is
  \begin{equation}\label{eq:b_coeff}
    \mathcal{B}=\frac{3\cdot 10^6}{2\alpha(2-\alpha)^2}+\frac{1}{4\alpha}.
  \end{equation}
 
The restrictions on $\mathcal C$ (cf. Equations~\eqref{eq:c_1 bound} and~\eqref{eqn: C const2}) will be
\begin{equation}
\label{eqn: C restrict}
\mathcal C>\frac{1}{2(2-\alpha)}\Big\{ 1+(2-\alpha)^2\mathfrak{s}\Big\}.
\end{equation}
(We note in passing that this is the first constraint for $\mathcal{B}$ in \eqref{eqn: B restrict}, after the transformation $\alpha\mapsto 2-\alpha$.)
Therefore, a valid choice for $\mathcal{C}$ is
  \begin{equation}\label{eq:c_coeff}
    \mathcal{C}=\frac{3\cdot 10^6}{2\alpha^2(2-\alpha)}+\frac{1}{4(2-\alpha)}.
  \end{equation}
  This choice implies the upper bound $\mathcal C<(2-\alpha) \mathfrak s$. 
  
The parameter $\mathcal A$ will need to satisfy (cf. Equations~\eqref{eqn: A const} and~\eqref{eqn: A const2}):
\begin{equation}
\label{eqn: A restrict}
\mathcal A>\frac{\alpha^2}{4}+\frac{\alpha \mathcal C}{2\mathfrak s}+2.
\end{equation}
This choice implies in particular
\begin{equation}
\label{eqn: A restrict 2}
\mathcal A^2>\alpha^2 + \frac{2\alpha \mathcal C}{\mathfrak s}+4.
\end{equation}
For example, one can take 
  \begin{equation}
    \mathcal{A}=10^3,\label{eq:a_coeff}
  \end{equation}
since, with the choices of $\mathcal{C}$ and $\mathfrak{s}$ above, and for $0<\alpha<2$, we have
\begin{equation}
\begin{aligned}
    \mathcal A=10^3&>4+\frac{\alpha(2-\alpha)}{8}(\alpha^2+3)\\
   % &=4+\frac{\alpha}{2}\frac{\alpha^2(2-\alpha)^2}{2}\frac{\alpha^2+12}{8(2-\alpha)\alpha^2}\\
    &>\frac{\alpha^2}{4}+\frac{\alpha \mathcal C}{2\mathfrak s}+2.
 \end{aligned}
 \end{equation}

Finally, the conditions on $\mathcal D$ will be as in \cite{argbourad20}
\begin{equation}
  \label{eqn: D restrict}
  \mathcal{D}=10^4.
\end{equation}

\subsection{Proof of Proposition \ref{prop: first}}
First, notice that
\[
H\cap G_1^c\subset  \ \ A_1^c \  \cup \ B_1^c \  \cup \ ( H\cap C_1^c\cap A_1\cap D_1)\  \cup \ ( D_1^c\cap A_1).
\]
We estimate the probability of the four events in the union on the right individually.

We first evaluate $A_1^c$:
\[
\PP(A_1^c)=\PP(|\widetilde S_{t_1}|>\mathcal{A} t_{1}).
\]
Equation \eqref{eqn: gaussian tail cplx} of the appendix is applicable with the choice $q=\lceil2 \mathcal A^2 t_1\rceil$, and implies that this is
\[
\ll \sqrt{t_{1}}\cdot \exp(-\mathcal{A}^2 t_1).
\]
Since $\mathcal{A}=10^3$, for some $\delta>0$ this is clearly
\begin{equation}\label{eq:a_1 bound}
  \ll\frac{e^{-\kappa^2 t}}{\sqrt t}t^{-\delta}.
\end{equation}

Turning to $B_1^c$, and applying Markov's inequality for some $q>1$ yields
\[
  \PP(B_1^c)\leq  \PP(S_{t_1}> U_1)\leq U_1^{-2q}\E[|S_{t_1}|^{2q}].
\]
Equation~\eqref{eqn: moment bound sqrt} then applies with $q=\lceil U_1^2/t_1\rceil$, giving
\begin{equation}
  \label{eqn: B const1}
  \PP(B_1^c)\ll \sqrt{t_1} e^{-U_1^2/t_1}\ll \frac{e^{-\kappa^2 t}}{\sqrt{t}}\cdot  t^{1+\kappa^2\mathfrak{s}-2\kappa \mathcal{B}}.
\end{equation}
By the choice of $\mathcal{B}$ in Equation \eqref{eqn: B restrict}, one has $1+\alpha^2\mathfrak{s}-2\alpha\mathcal{B}<0$. Since $\kappa=\alpha+\oo(1)$ by Equation \eqref{eq:slope}, the above is
\begin{equation}\label{eq:b_1 bound}
  \ll \frac{e^{-\kappa^2 t}}{\sqrt{t}}\cdot  t^{-\delta}
\end{equation}
for some $\delta>0$, depending on $\alpha$ and different from \eqref{eq:a_1 bound}.

To evaluate $ H\cap C_1^c \cap A_1\cap D_1$, we require the following lemma, proved in Section \ref{sect: lemma 4}. 
\begin{lem}
\label{lemma_4 easy}
  For $w$ with $|w|\leq 4 t_1$, we have
   \begin{equation}
  \label{eqn: lemma 4 super easy}
   \E\left[|\zeta \mathcal{M}_{1}|^4\ \mathbf{1}\left(S_{t_1}\in(w,w+1]\right)\right]
      \ll e^{4(t-t_{1})}\cdot \frac{e^{-w^2/t_1}}{\sqrt{t}}.
\end{equation}
\end{lem}

Let us explain the intuition behind the result. 
One should think of $\zeta \mathcal{M}_{1}$ as a random Euler product involving primes larger than $\exp(e^{t_{1}})$.
Furthermore, Selberg's result suggests its logarithm should be distributed like a Gaussian random variable of variance $t-t_{1}$. 
This explains the first factor $e^{4(t-t_{1})}$ as the contribution from the  moment generating function of such a variable. 
As explained in Section \ref{sect: lemmas}, the indicator function can be approximated by a suitable Dirichlet polynomial involving primes less than $S_{t_1}$. 
Since primes should behave independently, it is not surprising to see the decoupling between the factors. Most importantly, we obtain a Gaussian behavior for the variable $S_{t_1}$ in a large deviation regime.

The estimate $\PP(H\cap C_1^c\cap A_1\cap D_1 )$ is done by first partitioning on the value of $S_{t_1}$ using the restrictions given by $A_1$ and $C_1^c$:
\begin{align*}
\PP(H\cap C_1^c\cap A_1\cap D_1)&\leq\sum_{-\mathcal A(t_1-t_0)<u< L_1} \PP(\{S_{t_1}\in (u,u+1], |\zeta|>e^V\}\cap D_1)\\
&\leq\sum_{-\mathcal A (t_1-t_0)<u< L_1} \PP(\{S_{t_1}\in (u,u+1], |\zeta e^{-S_{t_1}}|>e^{V-u-1}\}\cap D_1),
\end{align*}
where we recall that $V=\kappa t$. The event $D_1$ implies that
\[
 |\zeta e^{-S_{t_1}}|\leq 2 |\zeta \mathcal M_1|+ e^{-\mathcal D(t-t_0)}.
\]
Therefore if $|\zeta e^{-S_{t_1}}|>e^{V-u-1}$, then it must be that either 
\[
 2|\zeta \mathcal M_1|>\frac{1}{2}e^{V-u-1}
\]
or
\[
e^{-\mathcal{D}(t-t_0)}>\frac{1}{2}e^{V-u-1}.
\]
The latter case is impossible, since it the exponent on the left side is negative, whereas on the right side we have on the range of $u$
\[
V-u-1> \kappa t-\kappa t_1 +\mathcal C\log t-1>0.
\]
This implies that
\begin{align*}
\PP(A_1\cap D_1 \cap H\cap C_1^c)&\leq \sum_{-\mathcal A (t_1-t_0)<u< L_1} \PP(|\zeta \mathcal M_1|>\frac{1}{100}e^{V-u}\cap\{S_{t_1}\in (u,u+1]\} )\\
& \ll \sum_{u< L_1} e^{-4(V-u)}\cdot \E[|\zeta \mathcal M_1|^4 {\mathbf 1}(S_{t_1}\in (u,u+1])].
\end{align*}
The sum over $u<0$ is $\ll e^{-4V}\cdot e^{4(t-t_1)}$ which is much smaller than $\frac{e^{-\kappa^2 t} }{\sqrt{t}} t^{-\delta}$ for the range of $V$ considered. 
Lemma \ref{lemma_4 easy} can be applied on the range $0\leq u<L_1$
This gives
\[
  \ll e^{4(t-t_1)}\sum_{u< L_1} e^{-4(V-u)}  \frac{e^{-u^2/t_1}}{\sqrt{t_1}}.
\]
After the change of variable $w=\kappa t_1-u$, this becomes 
\begin{equation}
  \label{eq:c_1 bound}
  \begin{aligned}
    e^{(4-4\kappa)(t-t_1)}\sum_{u< L_1} e^{-4(\kappa t_1-u)}  \frac{e^{-u^2/t_1}}{\sqrt{t_1}}
    &\ll\frac{e^{-\kappa^2 t_1}}{\sqrt{t_1}}e^{(4-4\kappa)(t-t_1)} \sum_{w>\mathcal C\log t} e^{-(4-2\kappa)w}\\ %bounded since dropped $w^2$
    &\ll \frac{e^{-\kappa^2 t} }{\sqrt{t}}\cdot t^{\mathfrak s\kappa^2 + \mathfrak s (4-4\kappa)-2(2-\kappa)\mathcal C }\\
    &\ll \frac{e^{-\kappa^2 t} }{\sqrt{t}} t^{-\delta},
  \end{aligned}
\end{equation}
for some $\delta>0$, by the choice of $\mathcal C$ in \eqref{eqn: C restrict}.

Finally we estimate $D_1^c$. In order to proceed, we need the following lemma.
The proof follows by expressing $e^{-(S_{t_{\ell+1}}-S_{t_{\ell}})}$ in terms of an Euler product, and by bounding the contribution of integers $m$ with $\Omega_\ell(m)>(t_\ell-t_{\ell-1})^{10^5}$ using Rankin's trick.
\begin{lem}[Lemma 23 in \cite{argbourad20}]
\label{lemma_23}
  Suppose $\ell\geq 0$ and that $|\widetilde{S}_{t_{\ell+1}}-\widetilde{S}_{t_\ell}|\leq 10^3(t_{\ell+1}-t_{\ell})$.  Then we have
  \[e^{-(S_{t_{\ell+1}}-S_{t_{\ell}})}\leq (1+e^{-t_{\ell}})|\mathcal{M}_{\ell+1}|+e^{-10^5(t_{\ell+1}-t_{\ell})}.\]
\end{lem}
Now, observe that the event $A_1\cap \{|\zeta|\leq e^{2 t}\}$ is contained in $A_1\cap D_1$. Indeed, since $|\widetilde{S}_{t_{1}}-\widetilde{S}_{t_0}|\leq 10^3(t_{1}-t_{0})$ on $A_1$, Lemma~\ref{lemma_23} implies
\[
|\zeta e^{-(S_{t_1}-S_{t_0})}|\leq 2|\zeta\mathcal{M}_1|+|\zeta|e^{-10^5(t_1-t_0)}\leq  2|\zeta\mathcal{M}_1|+e^{2t-10^5(t_1-t_0)},\]
which implies $D_1$ since $\mathcal D=10^4$.
Hence, to estimate  $\PP(D_1^c\cap A_1)$, it suffices to estimate $\PP(|\zeta|>e^{2t})$:
\begin{equation}\label{eq:d_1 bound}
  \PP(D_1^c\cap A_1)\leq \PP(|\zeta|>e^{2t})\leq e^{-4 t}\E[|\zeta|^2]\ll \frac{1}{\sqrt{t}}e^{-\kappa^2 t}e^{-100t},
\end{equation}
since $\E[|\zeta|^2]\ll e^t$ [Theorem 2.41 in \cite{harlit18}].

Summarising, we have by a union bound and successively applying Equations~\eqref{eq:a_1 bound},~\eqref{eq:b_1 bound},~\eqref{eq:c_1 bound}, and~\eqref{eq:d_1 bound},
\begin{align*}
  \PP(H\cap G_1^c)&\leq \PP(A_1^c)+\PP(B_1^c) + \PP(A_1\cap D_1\cap H\cap C_1^c) + \PP(A_1\cap D_1^c)\\
  %&\ll \frac{1}{\sqrt{t}}e^{-\kappa^2 t}(\log t)^{-2}+\frac{1}{\sqrt{t}}e^{-\kappa^2t}t^{-\delta}+ \frac{1}{\sqrt{t}}e^{-\kappa^2 t}(\log t)^{-\delta}+\frac{1}{\sqrt{t}}e^{-\kappa^2 t}e^{-100t}\\
  &\ll \frac{1}{\sqrt{t}}e^{-\kappa^2 t}t^{-\delta},
\end{align*}
for some $\delta>0$ dependent on $\alpha$. 

\subsection{Proof of Proposition \ref{prop: induction}}

Notice that
\[
H\cap G_\ell \cap  G_{\ell+1}^c \subset \ (A_{\ell+1}^c  \cap  G_\ell) \ \cup\  (B_{\ell+1}^c \cap G_\ell) \ \cup \ (H \cap C_{\ell+1}^c \cap A_{\ell+1}  \cap D_{\ell+1} \cap G_\ell) \ \cup \ ( D_{\ell+1}^c\cap A_{\ell+1}\cap G_\ell ).
\]
The probability of each event in the union on the right side are now evaluated.
In order to handle the event involving $A_{\ell+1}^c$ we will need the following lemma, proved in Section \ref{sect: lemmas}.

\begin{lem}\label{lemma_3}
  Let $\ell\geq 1$ be such that $10^6(t-t_\ell)^{10^5}e^{t_{\ell+1}}\leq \frac{1}{100}e^t$.  Let $\mathcal{Q}$ be a Dirichlet polynomial of length $N\leq \exp(\frac{1}{100}e^t)$, supported on integers all of whose prime factors are greater than $\exp(e^{t_{\ell}})$.  Then for $w\in[L_{\ell},U_{\ell}]$, we have
\[
    \mathbb{E}\left[|\mathcal{Q}(\tfrac{1}{2}+\ii\tau)|^2 \mathbf{1}\left(B_{\ell}\cap C_{\ell}\cap\{S_{t_{\ell}}\in(w,w+1]\}\right)\right]
      \ll \mathbb{E}\left[|\mathcal{Q}(\tfrac{1}{2}+\ii\tau)|^2\right]\cdot \frac{e^{-w^2/t_{\ell}} }{\sqrt{t_{\ell}}} .
      \]
\end{lem}
As in Lemma \ref{lemma_4}, the decoupling is due to the fact that the Dirichlet polynomials involve primes in different intervals. 
Though the events $B_{\ell}\cap C_{\ell}$ do not appear explicitly in the result, their presence here is crucial to obtain the Gaussian behavior of $S_{t_{\ell}}$ in a large deviation regime.

We first show that $\ell\geq 1$
\[
\PP(A_{\ell+1}^c \cap G_\ell)\ll \frac{e^{-V^2/t}}{\sqrt{t}}\cdot (\log_{\ell} t)^{-\delta}.
\]
For any $q>1$, the probability $\PP(A_{\ell+1}^c \cap G_\ell)$ is smaller than
\[
\begin{aligned}
 \sum_{u\in [L_\ell, U_\ell]}\E\left[\frac{|\widetilde S_{t_{\ell+1}}-\widetilde S_{t_\ell}|^{2q}}{(\mathcal{A}(t_{\ell+1}-t_\ell))^{2q}}\mathbf{1}(B_\ell\cap C_\ell\cap\{S_{t_\ell}\in (u,u+1]\})\right].
\end{aligned}
\]
With the choice $q=\lceil 2\mathcal{A}^2(t_{\ell+1}-t_\ell)\rceil$, the polynomial $\mathcal Q=|\widetilde S_{t_{\ell+1}}-\widetilde S_{t_\ell}|^{2q}$ both satisfies the assumptions of Lemma \ref{lemma_3} and Lemma \ref{lem: Gaussian moments cplx}.
Therefore, the above is
\[
\begin{aligned}
&\ll  \sum_{u\in [L_\ell, U_\ell]} (t_{\ell+1}-t_\ell)^{1/2} e^{-{\mathcal{A}}^2(t_{\ell+1}-t_\ell) } \frac{e^{-u^2/t_\ell}}{\sqrt{t_\ell}}\\
&\ll (t_{\ell+1}-t_\ell)^{1/2} \cdot e^{-{\mathcal{A}}^2(t_{\ell+1}-t_\ell) } \cdot \frac{e^{-L_\ell^2/t_\ell}}{\sqrt{t_\ell}},
\end{aligned}
\]
where the last inequality is by estimating the sum over $u$ trivially. Since $L_\ell=\kappa t_\ell -\mathcal{C}\log_\ell t$, this is
\begin{equation}
  \label{eqn: A const}
  \begin{aligned}
    \ll \frac{e^{-\kappa^2t_\ell}}{\sqrt{t}}\cdot (\log_{\ell-1}t)^{-{\mathcal{A}}^2 \mathfrak{s}  + 2\kappa \mathcal{C}}
    \ll \frac{e^{-\kappa^2t }}{\sqrt{t}}\cdot (\log_{\ell-1}t)^{\kappa^2\mathfrak{s}-{\mathcal{A}}^2 \mathfrak{s}  + 2\kappa \mathcal{C}}.
  \end{aligned}
\end{equation}
The choice of parameters in Equation \eqref{eqn: A restrict 2} guarantees that the exponent is negative. 

Now we show that for $\ell\geq 1$, 
\[
\PP(B_{\ell+1}^c\cap G_\ell)\ll \frac{1}{\sqrt{t}}e^{-\kappa^2 t}\cdot (\log_\ell t)^{-\delta}.
\]
By partitioning on the position of $S_{t_\ell}$, we have
\begin{align*}
  \PP(B_{\ell+1}^c\cap G_\ell)&\ll \PP(B_{\ell+1}^c\cap B_\ell\cap C_\ell)\\
  &\ll \sum_{u\in [L_\ell,U_\ell]} \PP(\{S_{t_{\ell+1}}-S_{t_\ell}>U_{\ell+1}-u\}\cap\{S_{t_\ell}\in(u,u+1]\}\cap B_{\ell-1}\cap C_{\ell-1})\\
  &\ll \sum_{u\in [L_\ell,U_\ell]} \E\left[\frac{(S_{t_{\ell+1}}-S_{t_\ell})^{2q}}{(U_{\ell+1}-u)^{2q}}\mathbf{1}(S_{t_\ell}\in(u,u+1], S_{t_k}\in[L_k, U_k] \ \forall k<\ell)\right],
\end{align*}
where the final line holds for any $q>1$ by an application of Markov's inequality, provided that $U_{\ell+1}-U_\ell>0$. This holds by the choice of $\mathcal{B}$ and $\mathfrak{s}$
in Equations \eqref{eqn: s} and \eqref{eqn: B restrict}.

Choosing $q= \lceil(U_\ell-u)^2/(t_{\ell+1}-t_\ell)\rceil$, then the Dirichlet polynomial $(S_{t_{\ell+1}}-S_{t_\ell})^q$ has length at most $\exp(2q e^{t_{\ell+1}})$ so the the conditions of Lemma~\ref{lemma_3} and Lemma~\ref{lem: Gaussian moments real} are satisfied. An application of Equation \eqref{eqn: gaussian tail real} then yields
\begin{align*}
  \PP(B_{\ell+1}^c\cap G_\ell)&\ll \sum_{u\in[L_\ell,U_\ell]}e^{-\frac{(U_{\ell+1}-u)^2}{t_{\ell+1}-t_\ell}}\cdot \frac{e^{-u^2/t_{\ell}}}{\sqrt{t_{\ell}}}
  \ll \sqrt{t_{\ell+1}-t_\ell}\cdot \frac{e^{-U_{\ell+1}^2/t_\ell}}{\sqrt{t}}.
\end{align*}
The last bound follows by bounding the sum over $u$ by the Gaussian integral. 
Since $U_{\ell+1}=\kappa t_{\ell+1}+\mathcal{B}\log_{\ell+1} t$, this is bounded by
\begin{equation}
  \label{eqn: B const2}
  \ll  \frac{e^{-\kappa^2 t}}{\sqrt{t}}\cdot (\log_\ell t)^{1/2+\mathfrak{s}\kappa^2-2\kappa\mathcal{B}}.
\end{equation}
The choice of $\mathcal{B}$ in Equation \eqref{eqn: B restrict} ensures that $1/2+\mathfrak{s}\kappa^2-2\kappa\mathcal{B}<0$.

The next estimate is 
\[
\PP( H\cap C_{\ell+1}^c\cap A_{\ell+1}\cap D_{\ell+1} \cap G_\ell ).
\]
For this, we need a more detailed version of Lemma \ref{lemma_4 easy}.
\begin{lem}
\label{lemma_4}
  Let  $\ell\geq 1$
  such that $10^6(t-t_\ell)^{10^5}e^{t_{\ell+1}}\leq \frac{1}{100}e^t$. 
  For $w\in[L_{\ell},U_{\ell}]$, we have
   \begin{equation}
  \label{eqn: lemma 4 easy}
   \E\left[|\zeta \mathcal{M}_{1}\cdots\mathcal{M}_{\ell}|^4\ \mathbf{1}\left(B_\ell\cap C_\ell, S_{t_\ell}\in[w,w+1]\right)\right]
      \ll e^{4(t-t_{\ell})}\cdot \frac{e^{-w^2/t_\ell}}{\sqrt{t}}.
\end{equation}
  Moreover, let $\gamma(m)$ be a sequence of complex coefficients with $|\gamma(m)|\leq \exp(\frac{1}{1000}e^t)$ for all $m\geq 1$.  
  Set \[\mathcal{Q}_\ell=\sum_{\substack{p|m\implies \log\log p\in(t_\ell,t_{\ell+1}]\\\Omega_{\ell+1}(m)\leq (t_{\ell+1}-t_\ell)^{10^4}}}\frac{\gamma(m)}{m^{\frac{1}{2}+\ii\tau}}.\]
  We have 
  \begin{equation}
  \label{eqn: lemma 4}
   \E\left[|\zeta \mathcal{M}_{1}\cdots\mathcal{M}_{\ell+1}|^4\right.\left.|\mathcal Q_\ell|^2\mathbf{1}\left(B_\ell\cap C_\ell, S_{t_\ell}\in[w,w+1]\right)\right]
      \ll e^{4(t-t_{\ell+1})}\cdot \mathbb{E}\left[|\mathcal{Q}_\ell|^2\right]\cdot \frac{e^{-w^2/t_\ell}}{\sqrt{t}}.
\end{equation}
\end{lem}

We now partition the values of $S_{t_\ell}=u$ for $L_\ell\leq u\leq U_\ell$ (on the event $G_\ell)$ as well as the values of the increments $S_{t_{\ell+1}}-S_{t_\ell}=v$ with the restrictions $u+v<L_{\ell+1}$ (on the event $C_{\ell+1}^c$) and $|v|\leq \mathcal A(t_{\ell+1}-t_\ell)$ (on the event $A_{\ell+1}$). The above is then smaller than
\[
\sum_{\substack{u\in [L_\ell, U_\ell]\\ u+v\leq L_{\ell+1}\\ |v|\leq \mathcal A(t_{\ell+1}-t_\ell)}}
\PP(\{S_{t_\ell}\in (u,u+1],S_{t_{\ell+1}}-S_{t_\ell}\in (v,v+1], |\zeta e^{-S_{t_{\ell+1}}}|>e^{V-(u+v)}\}\cap B_\ell \cap C_\ell \cap D_{\ell+1}).
\]
The definition of the event $D_{\ell+1}$ and the fact that  $|\zeta e^{-S_{t_{\ell+1}}}|>e^{V-(u+v+1)}$ imply that either
\[
c_\ell|\zeta\mathcal{M}_1\cdots\mathcal{M}_{\ell+1}|>\frac{1}{2} e^{V-(u+v+1)},
\]
or
\[
e^{-\mathcal D(t-t_\ell)}>\frac{1}{2} e^{V-(u+v+1)}.
\]
Again, the last case cannot occur, since the exponent on the left side is negative whereas the one on the right is
\[
V-(u+v)-1> V-L_{\ell+1}-1=(\mathfrak{s} \kappa + \mathcal{C})\log_{\ell+1}t-1>0,
\]
for $1\leq \ell \leq \mathcal{L}-1$ since $\log_{\mathcal L}t>0$ by construction.
This reduces this estimate to 
\[
\begin{aligned}
&\sum_{\substack{u\in [L_\ell, U_\ell]\\ u+v\leq L_{\ell+1}\\ |v|\leq \mathcal A(t_{\ell+1}-t_\ell)}}
\PP(\{S_{t_\ell}\in (u,u+1], S_{t_{\ell+1}}-S_{t_\ell}\in (v,v+1], |\zeta \mathcal M_1\cdots \mathcal M_{\ell+1}|>\frac{1}{100}e^{V-(u+v)}\}\cap B_\ell \cap C_\ell)\\
%&\ll
%\sum_{\substack{u\in [L_\ell, U_\ell]\\ u+v\leq L_{\ell+1}\\ |v|\leq \mathcal A(t_{\ell+1}-t_\ell)}}
%e^{-4 V+4(u+v)}\E[|\zeta \mathcal M_1\dots \mathcal M_\ell||^4\1(B_\ell \cap C_\ell\cap \{S_{t_\ell}\in )
&\ll \sum_{\substack{u\in [L_\ell, U_\ell]\\ u+v\leq L_{\ell+1}\\ |v|\leq \mathcal A(t_{\ell+1}-t_\ell)}}
e^{-4 V+4(u+v)}\E\Big[ |\zeta \mathcal M_1\cdots \mathcal M_{\ell+1}|^4 \cdot \frac{|S_{t_{\ell+1}}-S_{t_\ell}|^{2q}}{|v|^{2q}} \1(B_\ell \cap C_\ell\cap \{S_{t_\ell}\in (u,u+1]\})\Big],
\end{aligned}
\]
by Markov's inequality with $q=\lceil |v|^2/(t_{\ell+1}-t_\ell)\rceil\leq \mathcal A^2(t_{\ell+1}-t_\ell)$. 
Applications of Lemma \ref{lemma_4}, Equation~\ref{eqn: lemma 4} with $\mathcal Q_\ell=(S_{t_{\ell+1}}-S_{t_\ell})^q$ and Equation \eqref{eqn: gaussian tail real} then implies that this is
\[
\ll \sum_{\substack{u\in [L_\ell, U_\ell]\\ u+v\leq L_{\ell+1}\\ |v|\leq \mathcal A(t_{\ell+1}-t_\ell)}}
e^{-4 V+4(u+v)}\cdot e^{4(t-t_{\ell+1})}\cdot e^{-v^2/(t_{\ell+1}-t_\ell)}\frac{e^{-u^2/t_\ell}}{\sqrt{t_\ell}}.
\]
The change of variables $\bar u=u-\kappa t_{\ell}$ and $\bar v=v-\kappa( t_{\ell+1}-t_{\ell})$ and dropping some conditions on the sum gives
\[
\ll \frac{e^{-\kappa^2 t_{\ell+1}}}{\sqrt{t}}e^{(4-4\kappa)(t-t_{\ell+1})}\cdot \sum_{\substack{\bar v\in\mathbb Z\\ \bar u+\bar v\leq -\mathcal C\log_{\ell+1}t}}
e^{(4-2\kappa)(\bar u+\bar v)} e^{-\bar v^2/(t_{\ell+1}-t_\ell)},
\]
where we dropped the term $e^{-\bar u^2/t_\ell}$ since it is of order one by the restriction on $\bar u$. 
It remains to sum over $\bar u+\bar v$ first, then do the Gaussian sum on $\bar v$ to get
\begin{equation}
  \label{eqn: C const2}
  \frac{e^{-\kappa^2 t}}{\sqrt{t}}\cdot (\log_\ell t)^{\mathfrak{s}(2-\kappa)^2-2(2-\kappa)\mathcal C +1/2}.
\end{equation}
Again, the last term is $(\log_\ell t)^{-\delta}$ by the choice of parameters in \eqref{eqn: C restrict}. 

Finally we consider $D_{\ell+1}^c\cap A_{\ell+1}\cap G_\ell$.
We claim that it is enough to evaluate
\[
  \PP(\{|\zeta   \mathcal{M}_1\cdots \mathcal{M}_\ell|> e^{\mathcal{\mathcal A}(t-t_\ell)}\}\cap G_\ell).
\]
To see this, it suffices to notice that $A_{\ell+1} \cap \{|\zeta   \mathcal{M}_1\cdots \mathcal{M}_\ell|\leq e^{\mathcal{A}(t-t_\ell)}\}\cap D_\ell$ is in $A_{\ell+1}\cap D_{\ell+1}$.
Indeed, on the event $A_{\ell+1} \cap \{|\zeta   \mathcal{M}_1\cdots \mathcal{M}_\ell|\leq e^{\mathcal{A}(t-t_\ell)}\}\cap D_\ell$, we have
\begin{align}
  |S_{t_{\ell+1}}-S_{t_\ell}|&\leq \mathcal{A}(t_{\ell+1}-t_\ell)\label{eqn: al+1 for dl+1}\\
  |\zeta e^{-S_{t_{\ell}}}|&\leq c_{\ell}|\zeta   \mathcal{M}_1\cdots \mathcal{M}_\ell| +e^{-\mathcal{D}(t-t_{\ell-1})}\label{eqn: dl for dl+1}\\
    |\zeta   \mathcal{M}_1\cdots \mathcal{M}_\ell|&\leq e^{\mathcal{A}(t-t_\ell)} \label{eqn: dl zeta bound}.
\end{align}
Equations \eqref{eqn: al+1 for dl+1} and \eqref{eqn: dl for dl+1} imply that
\begin{align*}
  |\zeta e^{-S_{t_{\ell+1}}}|
  &\leq \left(c_\ell|\zeta   \mathcal{M}_1\cdots \mathcal{M}_{\ell}|+e^{-10^4(t-t_{\ell-1})}\right)e^{-(S_{t_{\ell+1}}-S_{t_\ell})}\\
  %&= c_\ell|\zeta \mathcal{M}_{\ell}|e^{-(S_{t_{\ell+1}}-S_{t_\ell})}+e^{-\mathcal{E}(t-t_{\ell-1})}e^{-(S_{t_{\ell+1}}-S_{t_\ell})}\\
  &\leq c_\ell|\zeta   \mathcal{M}_1\cdots \mathcal{M}_{\ell}|e^{-(S_{t_{\ell+1}}-S_{t_\ell})}+e^{-10^3(t-t_{\ell-1})},
\end{align*}
for $t$ large enough.
Then, Lemma \ref{lemma_23} gives
\[
|\zeta e^{-S_{t_{\ell+1}}}|\leq c_\ell|\zeta   \mathcal{M}_1\cdots \mathcal{M}_{\ell+1}|+c_\ell e^{10^3(t-t_\ell)-10^5(t_{\ell+1}-t_\ell)} +e^{-10^3(t-t_{\ell-1})}.
\]
We conclude that $D_{\ell+1}$ holds. 

It remains to estimate  $\PP(\{|\zeta \mathcal{M}_\ell|> e^{\mathcal{A}(t-t_\ell)}\}\cap G_\ell)$. We have by subsequent applications of Markov's inequality and Lemma~\ref{lemma_4}, Equation \eqref{eqn: lemma 4 easy},
\begin{align*}
  \PP(\{|\zeta \mathcal{M}_\ell|> e^{\mathcal{A}(t-t_\ell)}\}\cap G_\ell)
  &\ll e^{-4\mathcal{A}(t-t_\ell)}\E[|\zeta\mathcal{M}_{\ell}|^4\mathbf{1}(G_\ell)]\\
    &\ll e^{-4(\mathcal{A}-1)(t-t_\ell)}\frac{e^{-L_\ell^2/t_\ell}}{\sqrt{t_\ell}}\\
    %&\ll\frac{1}{\sqrt{t}}e^{-\kappa^2t}e^{-4(t-t_{\ell})(\mathcal{A}-1)}e^{\mathfrak{s}\kappa^2\log_\ell t+2\kappa\mathcal{C}\log_\ell t}\\
  &\ll\frac{1}{\sqrt{t}}e^{-\kappa^2t}e^{-4\mathfrak{s}(\mathcal{A}-1)\log_\ell t}e^{\mathfrak{s}\kappa^2\log_\ell t+2\kappa\mathcal{C}\log_\ell t}.
\end{align*}
We conclude that
\begin{equation}
  \label{eqn: A const2}
  \PP(D_{\ell+1}^c\cap A_{\ell+1}\cap G_\ell)\ll \frac{1}{\sqrt{t}}e^{-\kappa^2t}\cdot (\log_{\ell-1}t)^{\mathfrak{s}\kappa^2+2\kappa\mathcal{C}-4\mathfrak{s}(\mathcal{A}-1)}.
\end{equation}
The exponent is negative by Equation \eqref{eqn: A restrict}.

\subsection{Proof of Proposition \ref{prop: last}}

Finally we establish that
\[
\PP(H\cap G_\mathcal L)\ll \frac{1}{\sqrt{t}}e^{-\kappa^2 t}.
\]

After partitioning on the value of $S_{t_\mathcal L}$, applying Markov's inequality, and subsequently Lemma~\ref{lemma_4} we have
\[
\PP(H\cap G_\mathcal L)\ll \sum_{v\in[L_\mathcal L,U_\mathcal L]}e^{4(t-t_\mathcal L)}e^{4v}\frac{e^{-v^2/t_\mathcal L}}{\sqrt{t_\mathcal L}}.
\]
Applying the transformation $w=v-\kappa t_\mathcal L$, the probability is bounded by
\[
\ll \frac{1}{\sqrt{t}}e^{-\kappa^2 t}e^{(2-\kappa)^2\log_\mathcal L t}\sum_{-\mathcal{C}\log_\mathcal L t<w<\mathcal{B}\log_\mathcal L t}e^{2(2-\kappa)w}.
\]
Since $\alpha<2$, the sum is bounded by $\exp(2(2-\kappa)\mathcal{B}\log_\mathcal L t)$, so after grouping we find
\[
\ll \frac{1}{\sqrt{t}}e^{-\kappa^2 t}\cdot e^{(2-\kappa+2\mathcal{B})(2-\kappa)\log_\mathcal L t}.
\]
By the choice of $\mathcal L$, this is $\ll \frac{1}{\sqrt{t}}e^{-\kappa^2 t}$. 
%(i.e., $\mathcal L$ is such that $\log_\mathcal L t=\OO(1)$). 

\subsection{Proof of Lemma \ref{lemma_3}}
\label{sect: lemmas}

We express the event $B_\ell\cap C_\ell\cap\{S_{t_\ell}\in (w,w+1]\}$ in terms of the increments
\begin{equation}
  \label{eqn: Y}
  Y_j=S_{t_j}-S_{t_{j-1}}, \quad 1\leq j\leq \ell.
\end{equation}
The event implies that $S_{t_j}\in [L_j,U_j]$ for all $j$. We partition these intervals into subintervals of width $\Delta_j^{-1}$ where 
\[
\Delta_j=(t_j-t_{j-1}),
\]
so $\Delta_j\leq \mathfrak{s}\log_{j-1}t$ for $j>1$, and $\Delta_1$ is effectively $t_1=t-\mathfrak{s}\log t$. Note that $\Delta_j$ is of the same order as the variance of $Y_j$. Moreover, we have
\[
\sum_{j\geq 1}\Delta_j^{-1}\leq 1.
\]

Consider the set $\mathcal I$ of $\ell$-tuples $\mathbf u=(u_1,\dots, u_\ell)$ such that
\begin{equation}
  \label{eqn: I}
  \begin{aligned}
    \sum_{i=1}^{j} u_i\in[L_j-1,U_j+1],\quad j\leq \ell , \qquad  \sum_{i=1}^{\ell} u_i\in[w-1,w+1].
  \end{aligned}
\end{equation}
As a consequence of the definition, we have for all $j>1$
\begin{align*}
  L_j-1-(U_{j-1}+1)&\leq u_j\leq U_j+1-(L_{j-1}-1)
\end{align*}
which implies $|u_j| \leq ( \kappa\mathfrak s + \mathcal B+\mathcal C)\log_{j-1}t +2$. 
We will also shortly require the following estimate.  Since $\mathcal B<\alpha \mathfrak s$ and $\mathcal C<(2-\alpha) \mathfrak s$ (by \eqref{eq:b_coeff} and \eqref{eq:c_coeff}), we conclude from $\alpha<2$ that
\begin{equation}
  \label{eqn: u}
  |u_j|< 4\Delta_j+2.
\end{equation}
With these definitions, it is straightforward to check that we have the following inclusion of events
\begin{equation}
  \label{eqn: inclusion}
  B_\ell\cap C_\ell\cap \{S_{t_\ell}\in (w,w+1]\}\subset \bigcup_{\mathbf u\in \mathcal I}\{Y_j\in [u_j,u_j+\Delta_j^{-1}], 1\leq j\leq \ell\}.
\end{equation}
In particular, this implies
\begin{equation}
  \label{eqn: increment}
  \1(B_\ell\cap C_\ell\cap \{S_{t_\ell}\in (w,w+1]\})\leq \sum_{\mathbf u\in \mathcal I}\prod_{j}\1(Y_j\in [u_j,u_j+\Delta_j^{-1}]).
\end{equation}
We first prove:
\begin{lem}
  \label{lem: D}
  In the above notation, we have for $A\geq 10$ and $j\leq \ell$,
  \begin{equation}
    \label{eqn: approx 1}
    \1(Y_j\in [u_j,u_j+\Delta_j^{-1}])\leq |\mathcal D_{\Delta_j, A}(Y_j-u_j)|^2(1+c e^{-\Delta_j^{A-1}}),
  \end{equation}
  where $c$ is an absolute constant and $\mathcal D_{\Delta_j, A}(Y_j-u_j)$ is a Dirichlet polynomial on integers $n$ whose prime factors are in $(\exp(e^{t_{j-1}}),\exp(e^{t_j})]$ with $\Omega(n)\leq \Delta_j^{10A}$. In particular, its length is less than than $\exp(2e^{t_j}\Delta_j^{10A})$.
\end{lem}
\begin{proof}
  Lemma 6 in \cite{argbourad20} states that for any $\Delta, A \geq 3 $, there exists an entire function $G_{\Delta, A}(x) \in L^2(\mathbb{R})$ such that for some absolute constant $c>0$:
  \begin{enumerate}
  \item the Fourier transform $\widehat{G}_{\Delta, A}(x)$ is supported on $[-\Delta^{2A}, \Delta^{2A}]$;
  \item 
    $0 \leq G_{\Delta, A}(x) \leq 1$
    for all $x \in \mathbb{R}$;
  \item
    $\mathbf{1}(x \in [0, \Delta^{-1}]) \leq G_{\Delta, A}(x) \cdot (1 + c e^{-\Delta^{A - 1}});$
  \item
    $G_{\Delta, A}(x) \leq \mathbf{1}(x \in [-\Delta^{-A/2} , \Delta^{-1} + \Delta^{-A/2}]) + c e^{-\Delta^{A - 1}}; $
\item 
  $\int_{\mathbb{R}} |\widehat{G}_{\Delta, A}(x)| \rd x \leq 2\Delta^{2A} .$
  \end{enumerate}
  From the property (3), we get
  \begin{equation}
    \label{eqn: Y1}
    \1(Y_j\in [u_j,u_j+\Delta_j^{-1}])\leq |G_{\Delta_j,A}(Y_j-u_j)|^2 (1 + c e^{-\Delta_j^{A - 1}}).
  \end{equation}
  Writing $G_{\Delta_j,A}$ in terms of its Fourier transform, we have by truncating the exponential at $\nu=\Delta_j^{10A}$ (this choice will be motivated by the estimate \eqref{eqn: error2} below):
  \begin{equation}
    \label{eqn: G to D}
    \begin{aligned}
      G_{\Delta_j,A}(x)&=\int_\R e^{2\pi \ii \xi x}\widehat G_{\Delta_j,A}(\xi)\rd \xi\\
      &=\sum_{k\leq \nu}\frac{(2\pi \ii x)^k}{k!} \int_\R \xi^k\widehat G_{\Delta_j,A}(\xi)\rd \xi+ \OO^*\Big(\frac{(2\pi)^\nu x^\nu}{\nu!} \int_\R \xi^\nu\widehat G_{\Delta_j,A}(\xi)\rd \xi \Big),
    \end{aligned}
  \end{equation}
  where $\OO^*$ means that implicit constant is smaller than $1$ in absolute value. 
  The polynomial term in~\eqref{eqn: G to D} is our definition of the polynomial $\mathcal D_{\Delta_j,A}(x)$ in~\eqref{eqn: approx 1}. Since the $Y_j$ is a sum over primes in $(\exp(e^{t_{j-1}}),\exp(e^{t_j})]$, it is clear that $\mathcal D_{\Delta_j,A}(Y_j-u_j)$ is a Dirichlet polynomial involving integers with prime factors in that interval and that its length is at most $\exp(2e^{t_j}\Delta_j^{10A})$. (The factor $2$ comes from the fact that $Y_j$ includes squares of primes.) It remains to estimate the error term. 
  Since Equation \eqref{eqn: approx 1} is trivial if $|Y_j-u_j|>\Delta_j^{-1}$, we assume without loss of generality that $|Y_j-u_j|\leq\Delta_j^{-1}$. Therefore the error term is
  \begin{equation}
    \label{eqn: error}
    \begin{aligned}
      \frac{(2\pi)^\nu}{\nu!} \int_\R \xi^\nu\widehat G_{\Delta_j,A}(\xi)\rd \xi 
      &\leq \frac{(2\pi)^\nu}{\nu!} \int_\R |\xi|^\nu|\widehat G_{\Delta_j,A}(\xi)|\rd \xi
      &\leq \frac{(2\pi)^\nu }{\nu!} \cdot 2\Delta_j^{2A(\nu+1)}\leq \frac{(100)^\nu}{\nu^{\nu}}\Delta_j^{3A\nu},
    \end{aligned}
  \end{equation}
  where we use properties (1) and (5) above. 
%assuming $\nu>1$.
  This is $e^{-\Delta^{4A}}$ for the choice $\nu=\Delta_j^{10A}$. Putting this back in \eqref{eqn: Y1} yields
  \[
  \1(Y_j\in [u_j,u_j+\Delta_j^{-1}])\leq |\mathcal D_{\Delta_j,A}(Y_j-u_j)+\OO^*(e^{-\Delta_j^{4A}})|^2\cdot (1 + c e^{-\Delta_j^{A - 1}}).
  \]
  The term $\OO^*(e^{-\Delta_j^{4A}})$ can be absorbed in the multiplicative error by adjusting $c$. 
  The choice $A\geq 10$ ensures a decay much better than Gaussian.
\end{proof}

It follows from Equation \eqref{eqn: increment} and Lemma \ref{lem: D} that
\begin{equation}
  \1(B_\ell\cap C_\ell\cap \{S_{t_\ell}\in (w,w+1]\})\leq \sum_{\mathbf u\in \mathcal I}\prod_{j}  |\mathcal D_{\Delta_j, A}(Y_j-u_j)|^2(1+c e^{-\Delta_j^{A-1}}).
\end{equation}
We choose $A=20$ for the rest of the proof.
The product over $j$ of $ |\mathcal D_{\Delta_j, A}(Y_j-u_j)|^2$ is a Dirichlet polynomial of length at most
\[
\exp\big(2\sum_{j=1}^\ell e^{t_j}\Delta_j^{10A}\big)\leq\exp\Big(2e^{t_\ell}\Delta_\ell^{10A}\sum_{j=1}^\ell \Big(\frac{\log_{\ell-1}t}{\log_{j-1}t}\Big)^{\mathfrak s-10A}\Big)\leq  \exp(2e^{t_\ell}\Delta_\ell^{10A}),
\]
since $\mathfrak s\geq 10^6>10 A$ by the choice of $\mathfrak s$ in \eqref{eqn: s} and the choice $A=20$.
The mean-value theorem for Dirichlet polynomials, see Lemma \ref{lem: mean value DP} (which applies by the assumption on $\ell$), implies
\begin{equation}
  \label{eqn: random model}
  \E\Big[\prod_{j}  |\mathcal D_{\Delta_j, A}(Y_j-u_j)|^2\Big]=(1+\oo(1))\E\Big[\prod_{j}  |\mathcal D_{\Delta_j, A}(\mathcal Y_j-u_j)|^2\Big],
\end{equation}
where $(\mathcal Y_j, j\leq \ell)$ are independent random variables of the form 
\begin{equation}
  \mathcal Y_j=\sum_{e^{t_{j-1}}<\log p \leq e^{t_j}} \frac{\cos \theta_p}{p^{1/2}}+ \frac{\cos^2 \theta_p}{2p},
\end{equation}
and $(\theta_p, p\text{ primes})$ are independent random variables uniform on $[0,2\pi]$. 
 It remains to estimate $\E\Big[ |\mathcal D_{\Delta_j, A}(\mathcal Y_j-u_j)|^2\Big]$ for each $j$.
\begin{lem}
  \label{lem: D to Y}
  With the above notation, we have for $j\leq \ell$ and an absolute constant $c>0$,
  \[
  \E\Big[ |\mathcal D_{\Delta_j, A}(\mathcal Y_j-u_j)|^2\Big]\leq \PP(\mathcal Y_j-u_j\in [-\Delta_j^{-A/2} , \Delta_j^{-1} + \Delta_j^{-A/2}]) + c e^{-\Delta_j^{A - 1}}.
  \]
\end{lem}

\begin{proof}
  The idea is to use the approximation with $G_{\Delta_j, A}$ in reverse. For this, it is necessary to re-introduce the error term in Equation \eqref{eqn: G to D}, assuming it is small enough.
  On the event $|\mathcal Y_j-u_j|\leq \Delta_j^{6A}$, the estimate \eqref{eqn: error} becomes
  \begin{equation}
    \label{eqn: error2}
    \begin{aligned}
      \frac{(2\pi)^\nu\Delta_j^{6A\nu}}{\nu!} \int_\R \xi^\nu\widehat G_{\Delta_j,A}(\xi)\rd \xi 
      &\leq \frac{(2\pi)^\nu }{\nu!} \cdot \Delta_j^{2A(4\nu+1)}\leq \frac{(100)^\nu}{\nu^{\nu}}\Delta_j^{9A\nu}.
    \end{aligned}
  \end{equation} 
  This is $e^{-\Delta_j^{4A}}$ for the choice $\nu=\Delta_j^{10A}$. On the event $|\mathcal Y_j-u_j|>\Delta_j^{6A}$, Cauchy-Schwarz inequality yields
  \[
  \E\Big[ |\mathcal D_{\Delta_j, A}(\mathcal Y_j-u_j)|^2\1(|\mathcal Y_j-u_j|>\Delta_j^{6A})\Big]\leq 
  \E\Big[ |\mathcal D_{\Delta_j, A}(\mathcal Y_j-u_j)|^4\Big]^{1/2}\cdot \PP\big(|\mathcal Y_j-u_j|>\Delta_j^{6A})\big)^{1/2}.
  \]
  The fourth moment of $ \E[|\mathcal D_{\Delta_j, A}(\mathcal Y_j(h)-u_j)|^4]$ is bounded by
  \begin{equation}
    \label{eqn: CS}
    \E\Big[ \Big ( \sum_{\ell \leq \Delta_j^{10 A}} \frac{(2\pi )^{\ell}}{\ell!}  2\Delta_j^{2A(\ell+1)}  (|\mathcal Y_j|+|u_j|)^{\ell} \Big )^4\Big ]  
    \ll  \Delta_j^{2A} \, \E[\exp( 9\pi   \Delta_j^{2A}  (| \mathcal Y_j|+4\Delta_j))] \ll e^{\Delta_j^{5A}},
  \end{equation}
  where we used Equation \eqref{eqn: u} and the fact that $\E[e^{\lambda\mathcal Y_j}]\ll \exp(\lambda^2\Delta_j)$ by Lemma \ref{lem: model MGF}.
  The probability is bounded by Chernoff's inequality using the same lemma
  \begin{equation}
    \label{eqn: chernoff}
    \PP\big(|\mathcal Y_j-u_j|>\Delta_j^{6A}\big)\ll \exp(-\frac{1}{2}\Delta_j^{6A}).
  \end{equation}
  Equations \eqref{eqn: CS} and \eqref{eqn: chernoff} together imply
  \[
  \E\Big[ |\mathcal D_{\Delta_j, A}(\mathcal Y_j-u_j)|^2\1(|\mathcal Y_j-u_j|>\Delta_j^{6A})\Big]\leq e^{-\frac{1}{8}\Delta_j^{6A}}.
  \]
  Altogether, we have shown
  \[
  \E\Big[ |\mathcal D_{\Delta_j, A}(\mathcal Y_j-u_j)|^2\Big]
  \leq \E[|G_{\Delta_j,A}(\mathcal Y_j-u_j)+\OO(e^{-\Delta_j^{4A}})|^2]+e^{-\frac{1}{8}\Delta_j^{6A}}.
  \]
  Since $G_{\Delta_j,A}$ is in $[0,1]$ the error inside the expectation can be made additive. 
  The statement of the lemma then follows from property (4) of the function $G_{\Delta_j,A}$.
\end{proof}

The proof of Lemma \ref{lemma_3} can now be concluded.
\begin{proof}[Proof of Lemma \ref{lemma_3}]
  Let's first notice that by a direct application of Berry-Esseen theorem, see Lemma \ref{lem: ABH}, we have for any $j\geq 2$
  \begin{equation}
    \label{eqn: Y to N}
    \PP(\mathcal Y_j-u_j\in [-\Delta_j^{-A/2} , \Delta_j^{-1} + \Delta_j^{-A/2}])=\PP(\mathcal N_j-u_j\in  [-\Delta_j^{-A/2} , \Delta_j^{-1} + \Delta_j^{-A/2}])+\OO(e^{-ce^{t_{j-1}}}).
  \end{equation}
  where $\mathcal N_j$ is a Gaussian random variable of mean $0$ with variance $\frac{1}{2}(t_j-t_{j-1})+\oo(1)$.
  For $j=1$, we use the less accurate estimate in Lemma \ref{le:saddlepoint}:
  \[
  \PP(\mathcal Y_1-u_1\in [-\Delta_1^{-A/2} , \Delta_1^{-1} + \Delta_1^{-A/2}])\ll \PP(\mathcal N_1-u_1\in  [-\Delta_1^{-A/2} , \Delta_1^{-1} + \Delta_1^{-A/2}]).
  \]
  Since $|u_j|<4\Delta_j$ by Equation~\eqref{eqn: u}, we have that $u_j\cdot \Delta_j^{-A/2}$ is very small, and therefore by using a Gaussian estimate, we get for all $j\geq 2$,
  \[
  \PP(\mathcal N_j-u_j\in  [-\Delta_j^{-A/2} , \Delta_j^{-1} + \Delta_j^{-A/2}])=\PP(\mathcal N_j-u_j\in [0,\Delta_j^{-1}])(1+\OO(\Delta_j)^{-A/4}).
  \]
  For $j=1$, the corresponding estimate holds with $\ll$ instead of $=$. 
  We also notice that the error term in \eqref{eqn: Y to N} is much smaller than the probability and can be absorbed in the multiplicative error above.
  Therefore we have shown for $j\geq 2$ that
  \[
  \PP(\mathcal Y_j-u_j\in [-\Delta_j^{-A/2} , \Delta_j^{-1} + \Delta_j^{-A/2}])=\PP(\mathcal N_j-u_j\in [0,\Delta_j^{-1}])(1+\OO(\Delta_j)^{-A/4}),
  \]
  and for $j=1$
  \[
  \PP(\mathcal Y_1-u_1\in [-\Delta_1^{-A/2} , \Delta_1^{-1} + \Delta_1^{-A/2}])\ll\PP(\mathcal N_1-u_1\in [0,\Delta_1^{-1}])(1+\OO(\Delta_1)^{-A/4}).
  \]
  Putting this estimate back in Equation \eqref{eqn: increment} using Lemmas \ref{lem: D} and \ref{lem: D to Y} (noticing again that the additive error in Lemma \ref{lem: D to Y} can be made multiplicative), it follows that
  \[
  \PP(B_\ell\cap C_\ell\cap \{S_{t_\ell}\in (w,w+1]\})
    \ll \sum_{\mathbf u \in \mathcal I} \prod_{j=1}^\ell \PP(\mathcal N_j\in [u_j,u_j+\Delta_j^{-1}])(1+\OO(\Delta_j)^{-A/4}).
  \]
  It remains to re-express the events in terms of the partial sums of $\mathcal N_j$, exactly as we did in Equation \eqref{eqn: inclusion} but in reverse. By the definition of $\mathcal I$ and the summability of $\Delta_j^{-1}$, we conclude that
  \[
  \PP(B_\ell\cap C_\ell\cap \{S_{t_\ell}\in (w,w+1]\})\ll \PP\big(\sum_{j=1}^{\ell}\mathcal N_j\in (w-1,w+2]\big).
  \]
  Here, we dropped the intermediate restrictions on the partial sums that are no longer needed. The right side is $\ll \frac{1}{\sqrt{t_\ell}}e^{-w^2/t_\ell}$ as claimed.

\end{proof}

\subsection{Proof of Lemma \ref{lemma_4 easy} and Lemma \ref{lemma_4}}
\label{sect: lemma 4}
We prove Equation \eqref{eqn: lemma 4}. The proof of Lemma \ref{lemma_4 easy} and of Equation \eqref{eqn: lemma 4 easy} are similar and simpler.
The proof follows closely the one of Lemma \ref{lemma_3} with an additional tool from \cite{argbourad20}.
Given $\ell \geq 1$, a Dirichlet polynomial $\mathcal{Q}$ is said to be degree-$e^{t_\ell}$ well-factorable if it can be expressed as
\[
\prod_{1 \leq \lambda \leq \ell} \mathcal{Q}_{\lambda}(s),\quad \text{ where } \quad  \mathcal{Q}_\lambda(s) = \sum_{\substack{p | m \implies \log p \in (e^{t_{\lambda - 1}}, e^{t_\lambda}] \\ \Omega_{\lambda}(m) \leq 10 (t_{\lambda} - t_{\lambda - 1})^{10^4}}} \frac{\gamma(m)}{m^s},
\]
and $\gamma$ are arbitrary coefficients such that $|\gamma(m)| \leq \exp(\tfrac{1}{500} e^t)$ for every $m \geq 1$.
We need the following twisted fourth moment estimate.
\begin{lem}[Lemma 9 in \cite{argbourad20}]
  \label{lemma_9}
  Let $\ell \geq 0$ be such that $\exp(10^6 (t_{\ell+1}- t_{\ell})^{10^5} e^{t_{\ell + 1}}) \leq \exp(\tfrac {1}{100} e^t)$. 
  Let $\mathcal{Q}$ be a degree-$e^{t_{\ell+1}}$ well-factorable Dirichlet polynomial. Then, we have
  \[
  \mathbb{E} \Big [ |\zeta \mathcal{M}_{1} \cdots  \mathcal{M}_{\ell+1}|^4
  \cdot |\mathcal{Q}|^2  \Big ] 
  \ll  e^{4(t-t_{\ell+1})} \, \mathbb{E} \Big [ |\mathcal{Q}|^2\Big ].
  \]
\end{lem}

\begin{proof}[Proof of Lemma \ref{lemma_4}]
  We proceed as in the proof of Lemma \ref{lemma_3} by approximating the indicator function by a Dirichlet polynomial. 
  More precisely, using Equations \eqref{eqn: increment} and \eqref{eqn: approx 1}, the left-hand side of \eqref{eqn: lemma 4} becomes 
  \[
  \ll\sum_{\mathbf u \in \mathcal I} \mathbb{E} \Big [ |\zeta \mathcal{M}_{1} \cdots  \mathcal{M}_{\ell+1}|^4|\mathcal{Q}_\ell|^2
    \prod_j   \mathcal D_{\Delta_j,A}(Y_j-u_j)|^2\Big ] .
  \]
  We choose $A=20$. 
  The polynomial $Q=Q_\ell \prod_j \mathcal D_{\Delta_j,A}(Y_j-u_j)$ is well-factorable, and $\mathcal{Q}_\ell$ is as defined in the statement of Lemma~\ref{lemma_4}.
  Since the coefficients of $\mathcal D_{\Delta_j,A}$ are bounded by $\Delta_j^{2A(\nu+1)}$, the coefficients of $\mathcal Q$ are bounded by $\exp(\tfrac{1}{500} e^t)$.
  Moreover, its length is 
  \[
  \leq \exp(10 e^{t_{\ell+1}} (t_{\ell+1} - t_{\ell})^{10^4})\cdot \exp(2e^{t_\ell}\Delta_\ell^{200})<\exp(e^t/100),
  \]
  since $\mathfrak s\geq 10^6$ and by the assumption on $\ell$.
  This implies by Lemma \ref{lemma_9} that the above is
  \[
  \ll e^{4(t-t_{\ell+1})}\sum_{\mathbf u \in \mathcal I} \mathbb{E} \Big[|\mathcal{Q}_\ell |^2 \prod_j   \mathcal D_{\Delta_j,A}(Y_j-u_j)|^2\Big].
  \]
  The expectation splits by Lemma \ref{lem: splitting}. It remains to proceed as before from Equation \eqref{eqn: random model} to get Equation \eqref{eqn: lemma 4}.
\end{proof}

\section{Proofs of Corollaries}

\subsection{Proof of Corollary \ref{cor: fractional}}
\label{sect: fractional}
Consider the CDF of the random variable $\log|\zeta(1/2+\ii\tau)|$, i.e., $F(V)=\PP(\log|\zeta(1/2+\ii\tau)|\leq V)$. Write for short
\[
S(V)=\PP(\log|\zeta(1/2+\ii\tau)|>V).
\]
Recall that $\tau$ is distributed uniformly on $[T,2T]$, and we write $t=\log\log T$. Clearly, the  moments (cf. Equation~\eqref{eqn: M}) can be written as
\[
M_\beta=\int_{-\infty}^{+\infty} e^{\beta V} \rd F(V).
\]
Integration by parts yields
\begin{equation}
  \label{eqn: M expect}
  M_\beta=-e^{\beta V}S(V)\Big|_{-\infty}^{+\infty}+\int_{-\infty}^{+\infty}\beta e^{\beta V}S(V)\rd V.
\end{equation}
Since $S(V)$ is bounded by one, the boundary term at $-\infty$ is zero. Moreover, Markov's inequality with the fourth moment of zeta [Theorem B~\cite{ing26}] gives
\begin{equation}
  \label{eqn: markov}
  S(V)\leq \frac{1}{2\pi^2} e^{4t} e^{-4V} .
\end{equation}
In particular, this implies that the boundary term at $+\infty$ is zero for $\beta<4$. 
The contribution to negative $V$'s in the integral in Equation \eqref{eqn: M expect} is also negligible since
\begin{equation}
  \label{eqn: negative}
  \int_{-\infty}^{0}\beta e^{\beta V}S(V)\rd V\leq \int_{-\infty}^{0}\beta e^{\beta V}\rd V=1.
\end{equation}
It remains to estimate $\int_{0}^{+\infty}\beta e^{\beta V}S(V)\rd V$. Consider $\beta_-$ and $\beta_+$ such that $0<\beta_-<\beta< \beta_+<4$. 
These have to be chosen close enough to $0$ and to $4$ respectively.
It turns out that the choices 
\begin{align*}
  \beta_-&=\frac{\beta}{4}\\
  \beta_+&=\beta+\frac{3}{4}(4-\beta)=3+\frac{\beta}{4}
\end{align*}
are adequate. The dominant contribution to the $\beta$-moment comes from the interval $[\frac{\beta_-}{2}t, \frac{\beta_+}{2}t]$. Indeed, by Theorem \ref{thm: LD Selberg}, we have
\[
\int_{\frac{\beta_-}{2}t}^{\frac{\beta_+}{2}t}e^{\beta V}S(V)\rd V\ll \int_{\frac{\beta_-}{2}t}^{\frac{\beta_+}{2}t}e^{\beta V}\frac{e^{-V^2/t}}{\sqrt{t}}\rd V
=e^{\frac{\beta^2}{4}t}\int_{\frac{\beta_-}{2}t}^{\frac{\beta_+}{2}t}\frac{e^{-(\frac{\beta}{2}t-V)^2/t}}{\sqrt{t}}\rd V\ll e^{\frac{\beta^2}{4}t}.
\]
The contribution of the intervals $[0,\frac{\beta_-}{2}t]$ is less since it is smaller than
\begin{equation}
  \label{eqn: small beta}
  \int_0^{\frac{\beta_-}{2}t}\beta e^{\beta V}\rd V\leq e^{\frac{\beta^2 t}{8}},
\end{equation}
by the choice of $\beta_-$. For the interval $[\frac{\beta_+}{2}t, \infty]$, we use the bound \eqref{eqn: markov} to get that the contribution is
\[
\leq e^{4t}\int_{\frac{\beta_+}{2}t}^\infty e^{(\beta-4) V}\rd V\leq\frac{1}{4-\beta} e^{t(\frac{\beta\beta_+}{2}-2\beta_+ +4)}.
\]
This is $\ll e^{\beta^2t/4}$ by the choice of $\beta_+$. 

\subsection{Proof of Corollary \ref{cor: max}}
\label{sect: max}

We will require the following discretization result of~\cite{fargonhug07}.  Effectively, this shows that the maximum of concern in Corollary~\ref{cor: max} can be restricted to those $h$ lying $1/\log T$ apart. Corollary~\ref{cor: max} may also be deduced from a more general discretization result of~\cite{argouirad19}, applicable to Dirichlet polynomials.

\begin{lem}[Lemma 2.2 of~\cite{fargonhug07}]\label{lem: discretization}
  Let $t^*$ be such that $|\zeta(1/2+\ii t^*)|=\max_{t\in[T,2T]}|\zeta(1/2+\ii t)|$.  There is an absolute constant $A>0$ such that if $|t-t^*|<A/\log T$ then $2|\zeta(1/2+\ii t)|> |\zeta(1/2+\ii t^*)|$.
\end{lem}

Thus, as $\gamma$ ranges over a window of size $A/\log T$, the value of $|\zeta(1/2+\ii \gamma)|$ is close to the maximum within the window. Hence, we deduce via a union bound that, for some universal positive constant $C>0$,
\begin{equation}
  \label{eq: cor max discret}
  \PP\Big(\max_{|h|\leq \log^\theta T}|\zeta(1/2+\ii\gamma+\ii h)|>e^{V}\Big)\leq e^{(1+\theta)t}\cdot\PP\Big(|\zeta(1/2+\ii \gamma)| > \frac{1}{C}e^{V}\Big).
\end{equation}

Corollary~\ref{cor: max} now follows by setting $V=\sqrt{1+\theta}t-\frac{1}{4\sqrt{1+\theta}}\log t +y$ (for $y=\oo(t/\log t)$, $\theta\in[0,3)$), and applying Theorem~\ref{thm: LD Selberg}.
  
\subsection{Proof of Corollary \ref{cor: subcritical}}

{\bf \noindent Case $\beta\geq 0$}:
We write
\[
\mathcal Z_\beta(\tau)=\frac{1}{2e^{\theta t}}\int_{|h|\leq e^{\theta t}} |\zeta(1/2+\ii\tau +\ii h)|^\beta \rd h,
\]
i.e., the left-hand side of Equation~\eqref{eqn: subcritical} normalized by $2e^{\theta t}$ and with the identification $t=\log\log T$. 
The moment $\mathcal Z_\beta$ is a random variable dependent on $\tau$. From now on, we use the probabilistic convention and drop the dependence on $\tau$ from the notation. 
Consider also the (normalized) Lebesgue measure of high points in the interval $[-e^{\theta t}, e^{\theta t}]$ around $\tau$:
\[
\mathcal S(V)=\frac{1}{2e^{\theta t}} \m\{|h|\leq e^{\theta t}: \log |\zeta(1/2+\ii\tau+\ii h)|>V\}.
\]
Proceeding as in the proof of Corollary \ref{cor: fractional}, we have by integration by parts:
\[
\mathcal Z_\beta=- e^{\beta V} \mathcal S(V)\Big|_{-\infty}^{+\infty}+\beta\int_{-\infty}^{\infty} e^{\beta V} \mathcal S(V)\rd V .
\]
Again, since $\mathcal S(V)\leq 1$ for all $V$, we have that the boundary term at $V=-\infty$ is $0$. 

For $V=+\infty$, it is necessary to restrict the estimate to a good event.  Define
\begin{equation}
  \label{eqn: E}
  E=\left\{\max_{|h|\leq e^{\theta t}} \log |\zeta(1/2+\ii\tau+\ii h)|\leq m(t)+A\right\},
\end{equation}
where
\begin{equation}
  \label{eqn: m}
  m(t)=\sqrt{1+\theta}t-\frac{1}{4\sqrt{1+\theta}}\log t=\frac{\beta_c}{2}t-\frac{1}{2\beta_c}\log t,
\end{equation}
and $\beta_c=2\sqrt{1+\theta}$. In view of Corollary \ref{cor: max} with the choice $y=A$ (and since $\theta\in[0,3)$ by assumption), the probability of $E^c$ is
\begin{equation}
  \label{eqn: E^c}
  \PP(E^c)\ll e^{-\beta_c A}.
\end{equation}
This handles the upper limit $V=+\infty$.

On the event $E$, there are clearly no values of $V$ beyond $m(t)+A$.  Moreover, as in the proof of Corollary \ref{cor: fractional}, the contribution of negative values is of order one (cf. Equation \eqref{eqn: negative}). 
Finally, the bound \eqref{eqn: small beta} still holds.
The problem is therefore reduced to finding a good event on which to bound
\begin{equation}
  \label{eqn: Z reduction}
  \int_{\beta t/8}^{m(t)+A} e^{\beta V} \mathcal S(V) \rd V.
\end{equation}
The idea now is that $\mathcal S(V)$ should behave like $e^{-V^2/t-1/2\log t}$, thanks to Theorem \ref{thm: LD Selberg}.
In particular, as can be seen easily in the proof, the dominant contribution to the integral should come from $V$'s around $\beta t/2$. 
Hence, the specifics of the interval of integration do not matter much as long as it contains this optimizer. 
The main technical difficulty in implementing this idea is to control $\mathcal S(V)$ on a range of $V$ simultaneously.

Consider $(V_j, 1\leq j\leq J)$ the set of $V$'s in $[\tfrac{\beta}{8} t, m(t)+A]\cap \sqrt{t}\mathbb Z$, and additionally define $V_0=V_1-\sqrt{t}$ and $V_{J+1}=V_J+\sqrt{t}$.
(The choice of the mesh size $\sqrt{t}$ is informed by the typical fluctuation of $\log |\zeta|$.)
Define
\[
I_j=\int_{V_j}^{V_{j+1}} e^{\beta V} \mathcal S(V) \rd V, \quad 0\leq j\leq J.
\]
Consider the events
\begin{equation}
  \label{eqn: E_j}
  E_j=\left\{I_j \leq a_j \int_{V_j}^{V_{j+1}} e^{\beta V} \frac{e^{-V^2/t}}{\sqrt{t}} \rd V\right\},
\end{equation}
for a collection of $a_j$'s to be fixed later. 

We have $\PP(E_j^c)\ll a_j^{-1}$, since by linearity and Theorem \ref{thm: LD Selberg}
\begin{equation}
  \label{eqn: E[I]}
  \begin{aligned}
    \E[I_j]=\int_{V_j}^{V_{j+1}} e^{\beta V} \E[\mathcal S(V)] \rd V&\ll  \int_{V_j}^{V_{j+1}} e^{\beta V} \frac{e^{-V^2/t}}{\sqrt{t}} \rd V.
%=e^{\beta^2t/4}\cdot \int_{V_j}^{V_{j+1}} \frac{e^{-(\beta t/2-V)^2/t}}{\sqrt{t}} \rd V
\end{aligned}
\end{equation}
The good event to consider is
\[
G=E\cap\left(\bigcap_{j} E_j\right),
\]
so that by \eqref{eqn: E^c}
\begin{equation}
  \label{eq:prob_gee_complement}
  \PP(G^c)\ll\sum_j a_j^{-1}+ e^{-\beta_c A}.
\end{equation}
On the event $G$, we have 
\begin{equation}
  \label{eq:subcritical_gee_bound}
  \int_{\beta t/8}^{m(t)+A} e^{\beta V} \mathcal S(V) \rd V\leq e^{\beta^2 t/4}\sum_j a_j \int_{V_j}^{V_{j+1}} \frac{e^{-(\frac{\beta}{2}t-V)^2/t}}{\sqrt{t}} \rd V.
\end{equation}
Since the quadratic form is maximized at $\beta t/2$, we pick for $a_j$:
\[
a_j=
A\cdot 
\begin{cases}
\big(\frac{\beta}{2}\sqrt{t}-\frac{V_j}{\sqrt{t}}\big)^2+\frac{1}{100}\ &\text{if $V_j> \beta t/2$},\\
\big(\frac{\beta}{2}\sqrt{t}-\frac{V_{j+1}}{\sqrt{t}}\big)^2+\frac{1}{100}\ &\text{if $V_j\leq \beta t/2$ and $V_{j}<V_{j+1}\leq \beta t/2$}\\
\frac{1}{100}&\text{if $V_j\leq \beta t/2$ and $V_{j+1}>\beta t/2$}.
\end{cases}
\]
(The term $1/100$ is simply there to make sure $a_j$ is bounded away from $0$.)
This choice ensures that $a_j\leq A\big(\tfrac{1}{100}+(\frac{\beta}{2}\sqrt{t}-\frac{V}{\sqrt{t}})^2\big)$ for $V\in [V_j, V_{j+1}]$.

Thus, on one hand from Equation~\ref{eq:subcritical_gee_bound}, we have on $G$
\begin{align*}
  \int_{\beta t/8}^{m(t)+A} e^{\beta V} \mathcal S(V) \rd V
  &\leq A e^{\beta^2t/4}\int_{\beta t/8}^{m(t)+A}  \left(\frac{1}{100}+\left(\frac{\beta}{2}\sqrt{t}-\frac{V}{\sqrt{t}}\right)^2\right) \cdot \frac{e^{-(\frac{\beta}{2}t-V)^2/t}}{\sqrt{t}} \rd V\\
  &\leq A e^{\beta^2 t/4}\int_{\frac{(\beta-\beta_c)}{2}\sqrt{t}+\oo(1)}^{\frac{3\beta}{8}\sqrt{t}} \left(\frac{1}{100}+u^2\right) e^{-u^2}\rd u\\
  &\leq A e^{\beta^2 t/4},
\end{align*}
where the last bound follows by integrating over the whole line.
On the other hand, from Equation~\eqref{eq:prob_gee_complement} the probability of $G^c$ is
\[
\PP(G^c)\ll \sum_j a_j^{-1}+e^{-\beta_c A}\ll \frac{1}{A}.
\]
The $a_j$'s are summable since $V_j\in \sqrt{t}\mathbb Z$. This proves Equation \eqref{eqn: subcritical}.\\

{\bf \noindent Case $\beta>\beta_c$}:
We can use a reduction as in the previous case. We use the same event $E$ in \eqref{eqn: E} for the maximum. 
For a lower bound on the values of $V$, we take $\beta_c t/4$ since
\[
\int_0^{\beta_c t/4}e^{\beta V}\mathcal S(V) \rd V\leq e^{\tfrac{\beta_c}{4}\beta t},
\]
which is much smaller than the the desired bound. Therefore, it remains to estimate
\begin{equation}
  \label{eqn: Z reduction 2}
  \int_{\beta_c t/4}^{m(t)+A} e^{\beta V} \mathcal S(V) \rd V.
\end{equation}
The partitioning of the interval of integration is more delicate as it is close to the level of the maximum. A mesh size of $1$ instead of $\sqrt{t}$ is needed. 
More precisely, we take $(V_j, 1\leq j\leq J)$ to be $[\tfrac{\beta_c}{4}t, m(t)+A]\cap \mathbb Z$. 
The events $E_j$ are defined as in \eqref{eqn: E_j}. As before, we take $G=E\cap(\bigcap_j E_j)$.
The difference here is that the optimizer lies outside the interval, so the bound can be sharpened. 
On the event $G$, the above becomes
\[
\leq\sum_j a_j \int_{V_j}^{V_{j+1}} e^{\beta V} \frac{e^{-V^2/t}}{\sqrt{t}}\rd V.
\]
The change of variable $V=m(t)+y$ yields (with $y_j=V_j-m(t)$)
\begin{equation}
  \label{eqn: calc supercritical}
  \begin{aligned}
    e^{\beta m(t)}\sum_j a_j \int_{y_j}^{y_{j+1}} &e^{\beta y}\frac{e^{-m(t)^2/t}}{\sqrt{t}}e^{-\frac{2m(t)y}{t}}e^{-y^2/t}\rd y\\
    &\leq e^{\beta m(t)-(1+\theta)t} \sum_j a_j  \int_{y_j}^{y_{j+1}} e^{(\beta-\beta_c)y}e^{y\tfrac{(\log t)^2}{4\beta_c^2 t}}\rd y,
  \end{aligned}
\end{equation}
since $m(t)^2=(1+\theta)t -\frac{1}{2}\log t + \tfrac{(\log t)^2}{4\beta_c^2}$ and $e^{-y^2/t}\leq 1$. 
We pick $a_j=A(1+y_j^2)$ if $y_j$ is positive, and $a_j=A(1+y_{j+1}^2)$ if $y_{j+1}$ is negative. If $y_j<0<y_{j+1}$ then set $a_j=A$.  This choice ensures that $a_j\leq A(2+y^2)$ for $y\in [y_j, y_{j+1}]$, the term $2$ taking care of the values close to $0$.

This gives that Equation~\eqref{eqn: calc supercritical} is bounded by
\[
\begin{aligned}
&\leq A e^{\beta m(t)-(1+\theta)t} \int_{-\infty}^{A} (2+y^2)e^{(\beta-\beta_c)y}e^{y\tfrac{(\log t)^2}{4\beta_c^2 t}}\rd y\\
&\leq \Big(\tfrac{2A}{(\beta-\beta_c)^3}+\tfrac{A(A^2+2)}{\beta-\beta_c}\Big)e^{(\beta-\beta_c+1)A}\cdot e^{\beta m(t)-(1+\theta)t}
%=\tfrac{A^3e^{(\beta-\beta_c+1)A}}{\beta-\beta_c} e^{\tfrac{\beta_c}{2}\beta t-(1+\theta)t}\cdot t^{-\tfrac{\beta}{2\beta_c}},
\end{aligned}
\]
since $e^{y\tfrac{(\log t)^2}{4\beta_c^2 t}}\leq e^A$, and by direct integration of $(2+y^2)e^{(\beta-\beta_c)y}$. 
The probability of $G^c$ is then
$$
\PP(G^c)=e^{-\beta_c A}+\sum_{j}a_j^{-1}\ll \frac{1}{A}.
$$
This proves the corollary in the case $\beta>\beta_c$. 

\begin{rem}[Case $\beta=\beta_c$]
{\rm
Since it is possible to improve the bound \eqref{eqn: subcritical} in the range $\beta>\beta_c$, one might hope to do the same at $\beta=\beta_c$. 
This is possible in the case $\theta=0$, as discussed in the next section, but it is not expected to be possible for $\theta>0$.
Indeed, in this range of $\theta$, the above proof should be optimal. In fact, Equation \eqref{eqn: supercritical} would become (dropping the $a_j$'s for simplicity)
\begin{equation}
\label{eqn: calc critical}
\begin{aligned}
&e^{\beta_c m(t)} \int_{-\infty}^{A} e^{\beta y}\frac{e^{-m(t)^2/t}}{\sqrt{t}}e^{-\frac{2m(t)y}{t}}e^{-y^2/t}\rd t
\leq e^A\cdot e^{\frac{\beta_c^2}{4}t}  \int_{-\infty}^{A} \frac{e^{-y^2/t}}{\sqrt{t}}\rd y.
\end{aligned}
\end{equation}
This is because $e^{-(1+\theta)t}=e^{-\frac{\beta_c^2}{4}t}$ and $t^{-\frac{\beta}{2\beta_c}}=t^{-1/2}$.
The integral is now finite, so one recovers the bound \eqref{eqn: subcritical} up to a factor of order one.
}
\end{rem}

\section{Relation to Theorem \ref{thm: Harper} for $\theta=0$}
\label{sect: critical}
We briefly explain an alternative approach to proving a sharp upper bound to the $\beta_c$-moment in the case $\theta=0$.
It is based on the measure of the level sets in the spirit of the proof of Corollary \ref{cor: subcritical}.

The deterministic level of the maximum is now by Equation \eqref{eqn: FHK}
\[
m(t)=t-\frac{3}{4}\log t=\frac{\beta_c}{2}t-\frac{3}{2\beta c}\log t.
\]
There is a factor $3$ in the logarithmic correction and not $1$ as in \eqref{eqn: m}.
The important observation is that the typical measure of the level sets $\mathcal S(m(t)+y)$ is no longer $e^{-t}e^{-2y}e^{-y^2/t}$ as for the case $\theta>0$.
In fact, the proof of \eqref{eqn: FHK} in \cite{argbourad20} also shows that 
\begin{equation}
\label{eqn: S critical}
\mathcal S(m(t)+y)\leq A e^{-t}\cdot  |y| e^{-2y}e^{-y^2/2t}, \quad |y|=\oo(t),
\end{equation}
except on an event of probability $A$.
This is what is expected from the study of the extreme values of log-correlated processes, see for example Theorem 1.1 and Lemma 4.2 in \cite{corharlou19}. We explain how the additional $y$ in the decay is responsible for the extra $1/\sqrt{t}$ factor in the size of the moment. 
The integral \eqref{eqn: Z reduction 2} with $\beta=\beta_c=2$ becomes
\begin{equation}
  \label{eqn: Z reduction critical}
  \begin{aligned}
    \int_{ t/2}^{m(t)+A} e^{2 V} \mathcal S(V) \rd V
    &\leq A\frac{e^{2t}}{t^{3/2}}\int_{-t/2 +\tfrac{3}{4}\log t}^Ae^{-t}|y|e^{-y^2/t}\rd y\\
    &=A\frac{e^{t}}{t^{1/2}}\int_{-t/2 +\tfrac{3}{4}\log t}^A\frac{|y|}{\sqrt{t}}\frac{e^{-y^2/t}}{\sqrt{t}}\rd y\\
    &=A\frac{e^{t}}{t^{1/2}}\int_{-\sqrt{t}/2 +\oo(1)}^{A/\sqrt{t}}|u|e^{-u^2}\rd u.
  \end{aligned}
\end{equation}
The last integral is of order one. 
%Now, the last integral is of order one, thanks to the second $\sqrt{t}$ factor normalizing the $y$-factor coming from \eqref{eqn: S critical}.
At criticality, there is now an extra factor $1/\sqrt{t}$ coming from $t^{3/2}$ that is left, thereby giving the overall magnitude of $\frac{e^{t}}{t^{1/2}}$ for the moment.
It is also important to observe that, because of the $\sqrt{t}$-normalization in the integral, it is not necessary to know the level of the maximum up to order one as in Equation \eqref{eqn: FHK}.

\appendix

\section{Appendix}
The appendix gathers known results on moments of Dirichlet polynomials and probability estimates of random models.
\subsection{Moments of Dirichlet Polynomials}
\begin{lem}\label{lem: Gaussian moments cplx}
  Let $\widetilde{S}_j$ as in Equation \eqref{eqn: S tilde}.
  For any integers $t/2 \leq j \leq k$ and $2 q \leq e^{t - k}$, we have 
  \[
  \mathbb{E}[|\widetilde{S}_k - \widetilde{S}_j|^{2q}] \ll q! (k - j + 1)^{q}.
  \]
\end{lem}
\begin{proof}
  This is the content of \cite[Lemma 3]{sou09}.
\end{proof}
With the choice $q=\lceil \frac{V^2}{k-j+1}\rceil$, Markov's inequality, Lemma \ref{lem: Gaussian moments cplx} and Stirling's formula imply
\begin{equation}\label{eqn: gaussian tail cplx}
  \mathbb{P} \Big ( |\widetilde{S}_k - \widetilde{S}_j| > V \Big ) \ll \frac{V + 1}{(k - j )^{1/2}} \exp \Big ( - \frac{V^2}{k - j + 1} \Big ). 
\end{equation}

Lemma 16 of~\cite{argbourad20}] gives a more precise estimate for the moments of the real part $S_j$.
\begin{lem}\label{lem: Gaussian moments real}
  For any integers $t/2 \leq j < k$ and $2q \leq e^{t-k}$ we have
  \[
  \mathbb{E}[|S_k - S_j|^{2q}] \ll \frac{(2q)!}{2^q q!} \left(\frac{k - j}{2}\right)^{q}.
  \]
  Moreover, there exists $C>0$ such that for any $j< k$, and $2q\leq e^{t-k}$ such that
  \begin{equation}\label{eqn: moment bound sqrt}
    \mathbb{E}[|S_k - S_j|^{2q}] \ll \sqrt{q}\frac{(2q)!}{2^q q!} \left(\frac{k - j+C}{2}\right)^{q}.
  \end{equation}
\end{lem}
As in Equation \eqref{eqn: gaussian tail cplx}, one gets a Gaussian decay from Lemma~\ref{lem: Gaussian moments real} for the choice $q=\lceil \frac{V^2}{2(k-j+1)}\rceil$
\begin{equation}\label{eqn: gaussian tail real}
  \PP\left(|S_k-S_j|> V\right)\ll e^{-\frac{V^2}{k-j}},
\end{equation}
when $j>t/2$ and $V^2\leq \frac{k-j}{2}e^{t-k}$. 

We now explain the link between Dirichlet polynomials and the random model \eqref{eqn: random model}.
We consider the following general setup.
Let $(\theta_p, p \text{ prime})$ be a sequence of IID random variables, uniformly distributed on $[0,2\pi]$. 
For an integer $n$ with prime factorization $n = p_1^{\alpha_1} \ldots p_k^{\alpha_k}$ with $p_1, \ldots, p_k$ all distinct, define the random variable
\[
Z_n = \prod_{j = 1}^{k} \exp(\ii \alpha_j\theta_{p_j}). 
\]
By construction, we have the orthogonality relation $\mathbb{E}[Z_n \overline{Z}_m] = \mathbf{1}_{n = m}$. Therefore, for an arbitrary sequence $a(n)$ of complex numbers, the following holds
\[
\sum_{n \leq N} |a(n)|^2 = \mathbb{E} \Big [ \Big | \sum_{n \leq N} a(n) Z_n \Big |^2 \Big ]. 
\]
The expectation for the random variable is directly related to the mean-value of the square of Dirichlet polynomial, see \cite[Corollary 3]{MonVau07}. 

\begin{lem}
 \label{lem: mean value DP}
  We have
  \[
  \mathbb{E} \Big [ \Big | \sum_{n \leq N} a(n) n^{\ii \tau} \Big |^2 \Big ]  = \Big (1 + \OO \Big ( \frac{N}{T} \Big ) \Big ) \sum_{n \leq N} |a(n)|^2  = \Big ( 1 + \OO \Big ( \frac{N}{T} \Big ) \Big ) \mathbb{E} \Big [ \Big | \sum_{n \leq N} a(n) Z_n \Big |^2 \Big ].
  \]
\end{lem}
  
A direct consequence of Lemma \ref{lem: mean value DP} is the splitting of the expectation for Dirichlet polynomials involving different range of primes, see for example \cite[Lemma 14]{argbourad20}.
\begin{lem} \label{lem: splitting}
  Let
  \[
  A(s) = \sum_{\substack{n \leq N \\ p | n \implies p \leq w}} \frac{a(n)}{n^s} \text{ and } B(s) = \sum_{\substack{n \leq N \\ p | n \implies p > w}} \frac{b(n)}{n^s}
  \]
  be two Dirichlet polynomials with $N \leq T^{1/4}$. Then, we have
  \[
  \mathbb{E}[|A(\tfrac 12 + \ii\tau)|^2 \, |B(\tfrac 12 + \ii\tau)|^2 ] =(1+\OO(T^{-1/2}))\mathbb{E}[|A(\tfrac 12 + \ii\tau)|^2 ] \, \mathbb{E}[|B(\tfrac 12 + \ii\tau)|^2].
  \]
\end{lem}

\subsection{Estimates for the random model}
Recall the definition of the random model in Equation \eqref{eqn: random model}.
\begin{equation}
  \mathcal Y_j=\sum_{e^{t_{j-1}}<\log p \leq e^{t_j}} \frac{\cos \theta_p}{p^{1/2}}+ \frac{\cos^2 \theta_p}{2p}.
\end{equation}
The moment generating function is easily estimated using the independence between the $\theta_p$'s.
\begin{lem}
  \label{lem: model MGF}
  For $\lambda<\exp(\frac{1}{2}e^{t_j})$, we have
  \[
  \E[\exp(\lambda \mathcal Y_j)]\ll \exp\Big(\frac{\lambda^2}{4}(t_{j}-t_{j-1})\Big).
  \]
\end{lem}
\begin{proof}
  See for example \cite[Lemma 15]{argbourad20}.
\end{proof}
The comparison between the random model and the Gaussian model can be made more precise at the level of the probabilities. 
A version was proved in \cite[Proposition 2.11]{argbelhar17} using a Berry-Esseen estimate. See also \cite[Lemma 20]{argbourad20}.
\begin{lem}\label{lem: ABH}
  For $j\geq 2$, let $\mathcal N_j$ be a Gaussian random variable of mean $0$ and variance $\frac{1}{2}(t_j-t_{j-1})$.
  There exists a constant $c>0$ such that, for any interval $A$ and $j\geq 2$,
  \[
    {\mathbb P}\Big({\mathcal Y}_j \in A\Big)
    =\mathbb P\Big(\mathcal N_j\in A\Big)+\OO(e^{-c e^{j/2}}).
    \]
\end{lem}
In the case $j=1$ above, the variable $\mathcal Y_1$ is not asymptotically Gaussian because of the small primes. Nevertheless, the following estimate holds by a saddle-point method \cite[Lemma 18]{argbourad20}.
\begin{lem} \label{le:saddlepoint}
  Let $|v| \leq 100 r$.
  Then, for $r > 1000$ and for all $\Delta \geq 1$, we have
  \[
  \mathbb{P}(\mathcal Y_1 \in [v, v + \Delta^{-1}]) \asymp \frac{1}{\Delta} \cdot \frac{1}{\sqrt{r}} \exp \Big ( - \frac{v^2}{r} \Big ) .
  \]
\end{lem}

\bibliographystyle{alpha}
\bibliography{local_selberg.bib}

\end{document}